%% file: ASStmf3_final.tex
\documentclass[12pt]{amsart}

\input{preamble}

\title{The Adams spectral sequence for 3-local $\tmf$}
\author{D.~ Culver}\address{University of Illinois, Urbana-Champaign}\email{dculver@illinois.edu}

\begin{document}

\maketitle

\begin{abstract}
	The purpose of this article is to record the computation of the homotopy groups of 3-local $\tmf$ via the Adams spectral sequence. 
\end{abstract}

\tableofcontents

\input{intro}

\input{homologytmf}

\input{cohomologyGamma}

\input{AdamsE2}

\input{rational}

\input{AdamsDiffs}

\bibliographystyle{alpha}
\bibliography{ASSfortmf}

\end{document}

%% file: preamble.tex
\usepackage{amsmath}
\usepackage{amsthm}
\usepackage{amssymb}
\usepackage{lscape,xcolor}
\usepackage{graphicx}
\usepackage{mathrsfs}
\usepackage{stmaryrd}
\usepackage{verbatim}
\usepackage{rotating}
\usepackage{tikz-cd}
\usepackage{amsrefs}
\usepackage[bookmarks=false]{hyperref}
\usepackage{euscript}
\usepackage[colorinlistoftodos]{todonotes}
\usepackage{subcaption}

\graphicspath{ {figures/} }

\usepackage{sseq}
\usepackage{xcolor}
\definecolor{limegreen}{rgb}{0.2, 0.8, 0.2}
\definecolor{darkmagenta}{rgb}{0.55, 0.0, 0.55}
\definecolor{lavenderrose}{rgb}{0.91, 0.33, 0.5}
\definecolor{goldenpoppy}{rgb}{0.99, 0.76, 0.0}
\definecolor{seagreen}{rgb}{0.11, 0.35, 0.02}
\definecolor{maroon}{RGB}{128,0,0}
\definecolor{darkviolet}{RGB}{148,0,211}
\definecolor{twelve}{RGB}{100,100,170}
\definecolor{thirteen}{RGB}{100,150,50}
\definecolor{fourteen}{RGB}{200,0,0}
\definecolor{fifteen}{RGB}{0,200,0}
\definecolor{sixteen}{RGB}{0,0,200}
\definecolor{pinkpurple}{RGB}{200,0,200}
\definecolor{teal}{RGB}{0,200,200}

\allowdisplaybreaks[1]

\newcommand{\mmod}{\! \sslash \!}

\newcommand{\mr}[1]{\mathrm{#1}}

\newcommand{\br}[1]{\overline{#1}}

\newcommand{\Z}{\mathbb{Z}}

\newcommand{\Q}{\mathbb{Q}}

\newcommand{\F}{\mathbb{F}}

\newcommand{\tmf}{\mathrm{tmf}}

\def \HF2{\mr{H}\F_2}

\newcommand{\B}{\scr{B}}

\DeclareMathOperator{\Ext}{Ext}

\def \AA0{\br{A \mmod A(0)}_*}
\def \AA2{A\mmod A(2)_*}
\def \AE2{(A\mmod E(2))_*}
\renewcommand{\AE}[1]{(A\mmod E(#1))_*}

\def \E2E1{(E(2)\mmod E(1))_*}
\newcommand{\otau}{\overline{\tau}}

 \newtheorem{thm}[equation]{Theorem}
 \newtheorem{cor}[equation]{Corollary}
 \newtheorem{lem}[equation]{Lemma}
 \newtheorem{prop}[equation]{Proposition}

 \newtheorem*{thm*}{Theorem}
 \newtheorem*{cor*}{Corollary}
 \newtheorem*{lem*}{Lemma}
 \newtheorem*{prop*}{Proposition}
  \newtheorem*{not*}{Notation}

 \theoremstyle{definition}
 \newtheorem{defn}[equation]{Definition}
 \newtheorem{ex}[equation]{Example}
 
 \newtheorem{rmk}[equation]{Remark}

\newtheorem*{defn*}{Definition}
\newtheorem*{ex*}{Example}
\newtheorem*{exs*}{Examples}
\newtheorem*{rmk*}{Remark}
\newtheorem*{claim*}{Claim}

\numberwithin{equation}{section}
\numberwithin{figure}{section}

%% file: intro.tex

\section{Introduction}

In this paper, we will carry out a computation of the homotopy groups of $\tmf$. The homotopy groups of $\tmf$ have been known for quite some time. For example, the computation of $\pi_*\tmf$ was explicitly written up in \cite{Bauer_2008}, though it was known even earlier to Hopkins and Mahowald (cf. \cite{MahowaldHopkins}). The usual approach to calculating the homotopy of $\tmf$ and its variants is via the Adams-Novikov spectral sequence (also referred to as the descent spectral sequence in this context). One advantage of this approach is that the Adams-Novikov $E_2$-term can be computed using the theory of elliptic curves. 

However, there are occasions where one wants to know the Adams spectral sequence for computing the homotopy groups of a spectrum. That is, one may want to know the Adams $E_2$-term, all differentials, and all hidden extensions. The purpose of this paper is to record the Adams spectral sequence for 3-local topological modular forms.

We should mention that the analogous calculation at the prime 2 is being carried out by Rognes and Bruner (\cite{BrunerRognes}). It is the author's understanding that their interest in that spectral sequence stemmed from their work on the topological Hochschild homology of $\tmf$. We speculate knowing the Adams spectral sequence at the prime 3 might be useful for similar reasons. 

\subsubsection*{Conventions}

In this article, we will implicitly assume that all spectra are 3-complete. Thus $\tmf$ refers, from here on out, to the 3-completion of the spectrum of topological modular forms. We will always denote the mod 3 Eilenberg-MacLane spectrum by $H$. Given a Hopf algebra $\Gamma$ and a comodule $C$ over $\Gamma$, we will abbreviate $\Ext_{\Gamma}(\F_3, C)$ by $\Ext_{\Gamma}(C)$. In the case when $\Gamma = A_*$, we will write $\Ext(C)$. If $C = H_*X$ for a spectrum $X$, then we will abbreviate further to $\Ext(X)$. We will always employ Adams indexing unless specifically stated otherwise. We let $\zeta_n$ denote $\chi \xi_n$ and $\otau_n$ denote $\chi \tau_n$ in the dual Steenrod algebra. We will use the symbol $\dot{=}$ to indicate equality up to a multiplicative unit. Finally, for a spectrum $X$, we let $E_r(X)$ denote the $r$th page of the mod 3 Adams spectral sequence for $X$. 	

\subsection{Outline of the paper}

Recall that the Adams spectral sequence is a convergent spectral sequence of the form 
\[
\Ext_{A_*}(\F_3, H_*\tmf)\implies \pi_*\tmf^{\wedge}_3.
\]
Thus, a necessary input is $H_*\tmf$. This was determined, for example, in \cite{supplementary}, where Rezk shows there is a short exact sequence of comodules
\[
0\to \Sigma^8\B\to H_*\tmf \to \B\to 0
\]
where $\B$ is a certain subalgebra of the dual Steenrod algebra $A_*$. This is the starting point of our calculation. We view this short exact sequence as giving a multiplicative filtration of $H_*\tmf$ by comodules, yielding an algebraic spectral sequence 
\[
E_1^{*,*,*} = \Ext_{A_*}(E_0H_*\tmf)\cong \Ext_{A_*}(\B)\otimes E(b_4)\implies \Ext_{A_*}(H_*\tmf).
\]
In \textsection\ref{sec: homology of tmf} we recall these details and establish a change-of-rings formula for $\Ext_{A_*}(\B)$. In \textsection\ref{sec: cohomology of Gamma} we use a Cartan-Eilenberg spectral sequence to compute $\Ext_{A_*}(\B)$. An expert in these affairs can safely ignore this section. In \textsection\ref{sec: Adams E_2-term}, we determine the Adams $E_2$-term. In section \textsection \ref{sec: rational homotopy of tmf} we discuss the rational homotopy of $\tmf$ and its relationship to modular forms.  Finally, in \textsection\ref{sec: differentials}, we establish the Adams differentials and derive $\pi_*\tmf$.

\subsubsection*{Acknowledgements}

The author would like to thank Mark Behrens for encouraging him to write up this computation, as well as Bob Bruner for helpful discussions. He would also like to thank Hood Chatham for creating such a wonderful \LaTeX package for drawing spectral sequences as well as for assistance in drawing some of the charts in this paper. Finally, the author thanks an anonymous referee for carefully reading earlier drafts of this paper. They caught many typos and errors and suggested improvements to the exposition, resulting in a better paper. 

%% file: homologytmf.tex

\section{The mod 3 homology of $\tmf$}\label{sec: homology of tmf}

In this section we recall necessary facts about the mod 3 homology of $\tmf$. In \cite{supplementary}, it is shown that, as an algebra, the homology of $\tmf$ is given by 
\[
H_*\tmf\cong E(b_4)\otimes \B
\]
where $|b_4|=8$ and 
\[
\B:= \F_3[\zeta_1^3, \zeta_n\mid n\geq 2]\otimes E(\otau_n\mid n\geq 3)
\]
where the generators have the degrees
\begin{align*}
	|\xi_n| &= 2(3^n-1) & |\otau_n| &= 2\cdot 3^n-1 
\end{align*}
Here, the $\zeta_i$ are conjugate to Milnor's element $\xi_i$, and likewise $\otau_n$ is the conjugate of Milnor's $\tau_n$ \cite[Theorem 3.1.1]{greenbook}. One can easily check that $\B$ is a comodule algebra over $A_*$. Note also that for degree reasons, the class $b_4$ is an $A_*$-comodule primitive. Indeed, if 
\[
\alpha(b_4) = 1\otimes b_4 + \sum_i x_i'\otimes x_i''
\]
then the degrees of $x_i''$ are less than that of $b_4$. But $b_4$ is the lowest positive degree element of $H_*\tmf$. 

 Furthermore, Rezk shows that there is nontrivial extension of comodules

\begin{equation}\label{eqn: SES for Htmf}
	\begin{tikzcd}
		0\arrow[r] & \Sigma^8\B\arrow[r] & H_*\tmf \arrow[r] & \B\to 0.
	\end{tikzcd}
\end{equation}
Applying $\Ext_{A_*}(-)$ to this short exact sequence of comodules yields a long exact sequence in Ext. We regard this as a convergent spectral sequence

\[
\Ext(\Sigma^8\B)\oplus \Ext(\B)\implies \Ext(H_*\tmf).
\] 
The fact that \ref{eqn: SES for Htmf} is a nontrivial extension implies that this spectral sequence does not immediately collapse. Determining the differentials in this spectral sequence is the subject of section \ref{sec: Adams E_2-term}. Thus, it is apparent that we need to compute the Ext groups of $\B$. We will simplify this by establishing a change-of-rings formula. 

\begin{defn}
	Let $\Gamma$ be the Hopf algebra
	\[
	\Gamma:= A_*/(\zeta_1^3, \zeta_n, \otau_m\mid n\geq 2, m\geq 3)\cong \F_3[\zeta_1]/(\zeta_1^3)\otimes E(\otau_0, \otau_1, \otau_2)
	\]
	with the induced coproduct from the dual Steenrod algebra. 
\end{defn}

\begin{ex}
	In the dual Steenrod algebra, the coproduct on $\otau_2$ is given by 
	\[
	\psi(\otau_2) = \otau_2\otimes 1 + \otau_0\otimes \zeta_2+ \otau_1\otimes \zeta_1^3 + 1\otimes \otau_2.
	\]
	Thus, in $\Gamma$, $\otau_2$ is a Hopf algebra primitive. On the other hand, 
	\[
	\psi(\otau_1) = \otau_1\otimes 1 + \otau_0\otimes \zeta_1 + 1\otimes \otau_1. 
	\]
	Thus this Hopf algebra is not primitively generated.
\end{ex}

The proof of the following proposition is standard.

\begin{prop}
	There is an isomorphism
	\[
	\B\cong A_*\boxempty_\Gamma \F_3.
	\]
\end{prop}

We derive the following corollary from Theorem A1.3.12 of \cite{greenbook}.

\begin{cor}
	There is a change-of-rings isomorphism
	\[
	\Ext(\B)\cong \Ext_{\Gamma}(\F_3).
	\]
\end{cor}

Thus we must compute the cohomology of the Hopf algebra $\Gamma$. This is done in the next section.

%% file: cohomologyGamma.tex

\section{Computing the cohomology of $\Gamma$}\label{sec: cohomology of Gamma}

In the last section we showed that the Ext groups of $\B$ are $\Ext_{\Gamma}(\F_3)$. Since $\Gamma$ is a finite Hopf algebra, there is hope of computing its cohomology. Recall that $A(1)$ is the subalgebra of the Steenrod algebra generated by the Bockstein $\beta$ and $\mathcal{P}^1$. Its dual is 
\[
A(1)_*\cong \F_3[\zeta_1]/(\zeta_1^3)\otimes E(\otau_0, \otau_1).
\]
In paticular, $A(1)_*$ is a sub-Hopf algebra of $\Gamma$. The following proposition relies on the material in the first appendix of \cite{greenbook}. We recommend the reader look at Definition A1.1.15. The following lemma is easily checked. 

\begin{lem}
	\label{prop: hopf extension}
	The following 
	\[
	A(1)_*\to \Gamma \to E(\otau_2)
	\]
	is a cocentral extension of Hopf algebras over $\F_3$.
\end{lem}

When one has an extension of Hopf algebras, one can consider the Cartan-Eilenberg spectral sequence. In general, if 
\[
(D, \Phi)\to (A, \Gamma)\to (A, \Sigma)
\]
is an extension of Hopf algebroids, $N$ is a left comodule over $\Gamma$, then there is a natural convergent spectral sequence of the form
\[
E_2^{f,s,t}= \Ext^{f,t}_\Phi(D, \Ext^s_\Sigma(A, N))\implies \Ext^{f+s,t}_{\Gamma}(A, N).
\]
Here, $f$ denotes the filtration degree, $s$ is the cohomological degree, and $t$ is the internal degree. The differentials are of the form
\[
d_r: E_r^{f,s,t}\to E_r^{f+r,s-r+1, t}.
\]
See A1.3.14 and A1.3.15 of \cite{greenbook} for details on this spectral sequence. Applied to our extension of Hopf algebras with $N=\F_3$, this spectral sequence takes on the form
\begin{equation}\label{eq: CESS}
E_2^{f,s,t} = \Ext^{f,t}_{A(1)_*}(\F_3, \Ext^s_{E(\otau_2)}(\F_3, \F_3))\implies \Ext^{f+s,t}_{\Gamma}(\F_3).
\end{equation}
Since $E(\otau_2)$ is a primitively generated exterior Hopf algebra, we have that 
\[
\Ext_{E(\otau_2)}(\F_3)\cong \F_3[v_2]
\]
where the $(s,t)$-bidegree of $v_2$ is $(1, 17)$. Note that since $\F_3$ is a comodule algebra, the Cartan-Eilenberg spectral sequence is multiplicative. 

In order to determine the $E_2$-page of this spectral sequence, we need to understand the coaction of $A(1)_*$ on $\F_3[v_2]$. As $\F_3[v_2]$ is a comodule algebra over $A(1)_*$, it is enough to determine the coaction on $v_2$. 

\begin{lem}
	Under the canonical $A(1)_*$-coaction on $\Ext_{E(\otau_2)}(\F_3)$, the element $v_2$ is a comodule primitive. 
\end{lem}
\begin{proof}
Observe that the largest degree element of $A(1)_*$ is $\zeta_1^2\otau_0\otau_1$, which has degree 14. Since 
\[
\Ext_{E(\otau_2)}(\F_3)\cong \F_3[v_2], \,\,\,\, |v_2| = (1,17)
\]
the coaction on $v_2$ must be $1\otimes v_2$ for degree reasons.
\end{proof}

\begin{cor}
	The $E_2$-term of the Cartan-Eilenberg spectral sequence (CESS) is given by 
	\[
	E_2\cong \Ext_{A(1)_*}(\F_3)\otimes \F_3[v_2]
	\]
	where $v_2$ is in $(f,s,t)$-degree $(0,1, 17)$ and $\Ext^{a,t}_{A(1)_*}(\F_3)$ is in tridegree $(a, 0, t)$.
\end{cor}
Thus we must determine the cohomology of $A(1)_*$. The May spectral sequence can be used for this purpose. Later, we will need to use May's convergence theorem to give a proof for Lemma \ref{lem: Massey prod}, so we collect some details about the spectral sequence here. The reader is referred to \cite[3.2]{greenbook} for further details. 

This spectral sequence is obtained by putting a filtration on 
\[
	A(1)_* = P(\zeta_1)/(\zeta_1^3)\otimes E(\otau_0, \otau_1).
\]
This filtration is defined by assigning the generators of $A(1)_*$ the following \emph{May weight}.
	\begin{itemize}
		\item $MF(\otau_0) = MF(\zeta_1) = 1$,
		\item $MF(\otau_1) = 3$.
	\end{itemize}
	The associated graded of this filtration is given by 
	\[
	E^0A(1)_* = P(\zeta_1)/(\zeta_1^3)\otimes E(\otau_0, \otau_1)
	\]
	but now with $\zeta_1, \otau_0, \otau_1$ as primitive elements. This produces a filtration on the cobar complex for $A(1)_*$, resulting in the May spectral sequence
	\[
	E_1^{m,s,t}\implies \Ext_{A(1)_*}^{s,t}(\F_3),
	\]
	with the following $E_1$-term, 
	\[
	E_1= \Ext_{E^0A(1)_*}(\F_3)\cong \Ext_{P(\zeta_1)/\zeta_1^3}(\F_3)\,\otimes\, \Ext_{E(\otau_0, \otau_1)}(\F_3)\cong E(\alpha_1)\,\otimes\, P(v_0,v_1, \beta).
	\]
	Here we are using the fact that 
	\[
	\Ext_{E(\otau_0, \otau_1)}(\F_3)\cong P(v_0,v_1)
	\]
	since $\otau_0, \otau_1$ are primitive and that 
	\[
	\Ext_{P(\zeta_1)/\zeta_1^3}(\F_3)\cong E(\alpha_1)\otimes P(\beta)
	\]
	since $\zeta_1$ is primitive. The tri-degrees of these classes in the May spectral sequence are recorded below:
	\begin{enumerate}
		\item $|\alpha_1| = (1, 1,4)$,
		\item $|\beta| = (3, 2, 12)$,
		\item $|v_0| = (1,1,1)$, 
		\item $|v_1| = (3,1,5)$.
	\end{enumerate}
	Moreover, the $\Ext$-groups of primitively generated truncated polynomial algebras also have the following Massey product
	\[
	\beta = \langle \alpha_1, \alpha_1, \alpha_1\rangle,
	\]
	see \cite[Lemma 3.2.4]{greenbook}. Indeed, $\beta$ can be represented in the cobar complex for $P(\zeta_1)/\zeta_1^3$ by 
	\[
	[\zeta_1^2\mid \zeta_1] - [\zeta_1\mid \zeta_1^2]
	\]
	which is precisely the Massey product above. Finally, the coproduct on $A(1)_*$ gives the $d_1$-differential
	\[
	d_1(v_1) = v_0\alpha_1.
	\]
	The rest of the $d_1$-differentials are propagated from this one and the multiplicativity of the May spectral sequence.
	
\begin{prop}\label{prop: Ext of A(1)}
	The algebra $\Ext_{A(1)_*}(\F_3)$ is given by 
	\[
	\F_3[v_0,v_1^3,\beta]\otimes E(\alpha_1, \alpha_2)/(v_0\alpha_1, v_0\alpha_2, \alpha_1\alpha_2-v_0\beta)
	\]
	where the $(s,t)$-bidegrees of the generators are given by 
	\begin{itemize}
		\item $|\alpha_1|=(1,4)$,
		\item $|\beta| = (2, 12)$,
		\item $|\alpha_2| = (2, 9)$,
		\item $|v_0| = (1,1)$,
		\item $|v_1| = (1,5)$
	\end{itemize}
\end{prop}

A chart for this Ext group is given below. 

\begin{figure}[h]
	\includegraphics[width=\textwidth]{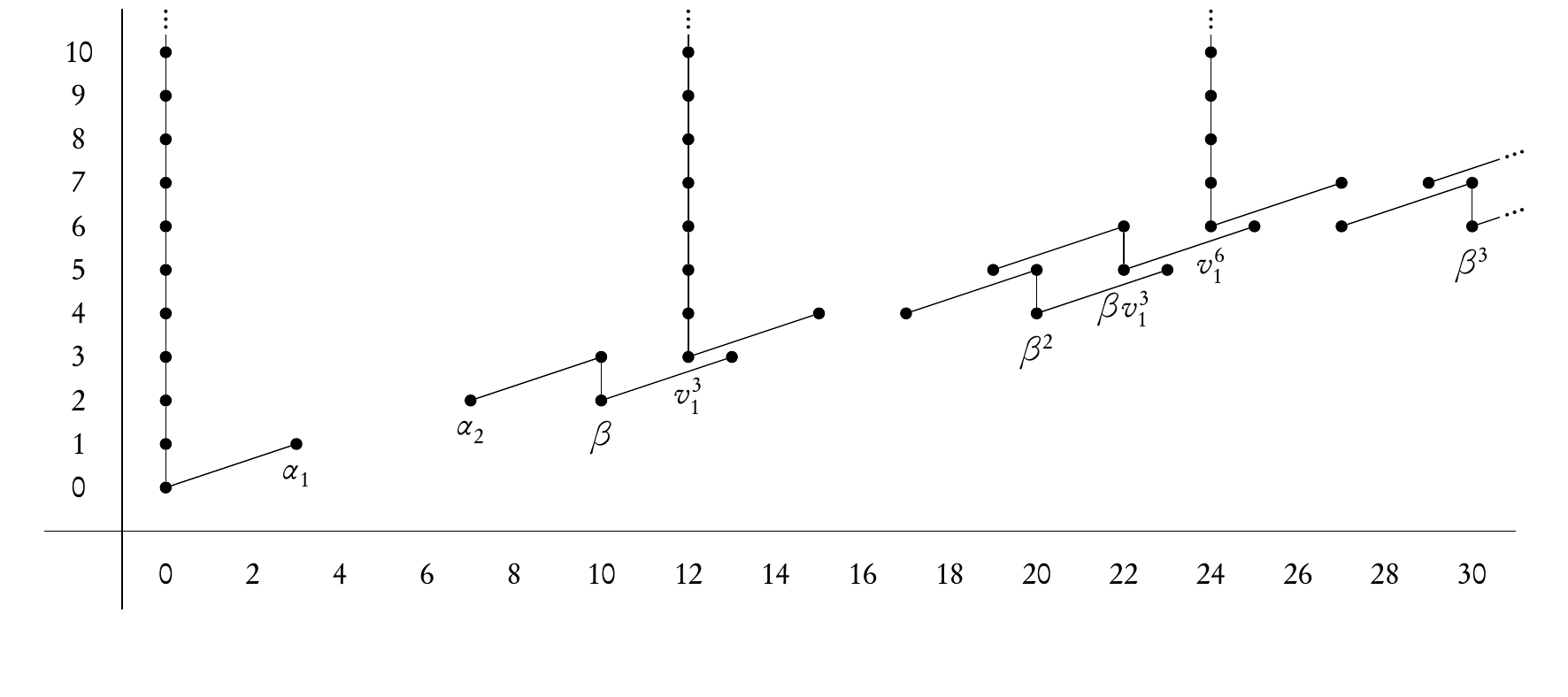}
	\centering 
\end{figure}


%
%
%
%
%
%

From now on, we will write $c_6$ for $v_1^3$. This is justified by Proposition \ref{prop: v_1^3 detects c6} below. For degree reasons, this spectral sequence collapses. Indeed, if we use $(t-s,s)$-indexing to depict the Cartan Eilenberg spectral sequence $E_2$-term, then a $d_r$-differential goes up vertically $s-r+1$-spaces. Since $\Ext_{A(1)_*}(\F_3)=0$ in degree $t-s=16$, it follows there cannot be any differentials on $v_2$.

\begin{cor}
	The cohomology of the Hopf algebra $\Gamma$ is given by 
	\[
	\F_3[v_0,c_6,v_2,\beta]\otimes E(\alpha_1, \alpha_2)/(v_0\alpha_1, v_0\alpha_2, v_0\alpha_2-v_0\beta).
	\]
\end{cor}
\begin{proof}
	We have already shown that this is the $E_\infty$-page. All that remains to be shown is that there are no hidden extensions. First, note that in \eqref{eq: CESS}, $f$ denotes the filtration degree. Thus $v_2$ is in filtration 0. The Cartan-Eilenberg spectral sequence arises from an increasing filtration on a cochain complex (see \cite[A1.3.14, A1.3.16]{greenbook} for details). Thus, if $x$ and $y$ are two classes, then a hidden extension from $xy$ to $z$ implies that the filtration of $z$ is larger than that of $xy$. Note also that there are no hidden extensions in the $\Ext_{A(1)_*}(\F_3)$-submodule generated by 1. This is seen, for example, by noticing that the map $A(1)_*\to \Gamma$ induces a map in $\Ext$, 
	\[
	\Ext_{A(1)_*}(\F_3)\to \Ext_\Gamma(\F_3).
	\]
	
	So, the only possible hidden extensions would involve products of classes in the ideal generated by $v_2$. Suppose $v_2^i x$ and $v_2^j y$ are classes with $x, y\in \Ext_{A(1)_*}(\F_3)$. Then if $v_2^ix\cdot v_2^j y = v_2^{i+j}xy=0$, there could be a hidden extension to a class in higher filtration, let's say there was an extension $v_2^{i+j}xy = v_2^k z$ with $z\in \Ext_{A(1)_*}(\F_3)$. Note that in the usual Adams indexing the bidegrees must agree. Let us name the tri-degrees of these classes,
	\begin{itemize}
		\item $|x| = (f_0, 0, t_0)$
		\item $|y| = (f_1, 0, t_1)$
		\item $|z| = (f_2, 0, t_2)$.
	\end{itemize}
	Then the tri-degree of $v_2^ix\cdot v_2^jy = v_2^{i+j}xy$ is
	\[
	(f_0+f_1, i+j, t_0+t_1+17(i+k))
	\]
	whereas the tri-degree for $v_2^kz$ is 
	\[
	(f_2, k, t_2+17k).
	\]
	These classes detect elements in $\Ext_{\Gamma}(\F_3)$ of bidegrees
	\[
	(f_0+f_1 + i+j, t_0+t_1+17(i+j))
	\]
	and 
	\[
	(f_2+k, t_2+17k)
	\]
	respectively. These bidegrees must be equal, but in order for this to be a hidden extension, we must have $f_0+f_1<f_2$. This implies that $i+j>k$. Since $v_2$ is not a zero divisor on $E_\infty$, it follows that it is not a zero divisor in $\Ext_{\Gamma}(\F_3)$. Thus, if we had the equality 
	\[
	v_2^{i+j}xy = v_2^k z
	\]
	in $\Ext_\Gamma$, then we would have 
	\[
	v_2^{i+j-k}xy = z
	\]
	in $\Ext_\Gamma$. However, on $E_\infty$, the only way we could have had $v_2^ix\cdot v_2^jy=0$ is if $xy=0$. Since $x, y\in \Ext_{A(1)_*}^{s,t}(\F_3)\cdot\{1\}$, it follows that $xy=0$ in $\Ext_{\Gamma}(\F_3)$. This implies that $z=0$. So there are no hidden extensions. 
\end{proof}

%% file: AdamsE2.tex

\section{Determining the Adams $E_2$-term}\label{sec: Adams E_2-term}

In this section we will determine the $E_2$-term of the Adams spectral sequence converging to $\pi_*\tmf$. The way this will be achieved is by applying the functor $\Ext(-)$ to the short exact sequence \eqref{eqn: SES for Htmf} to obtain a long exact sequence. Regarding this as a spectral sequence provides us with 
\[
E_1=\Ext(\Sigma^8\B)\oplus \Ext(\B)\implies \Ext(\tmf).
\]
For the purposes of this paper, we will refer to this spectral sequence as the \emph{algebraic spectral sequence}. 

\subsection{Algebraic differentials}


	The short exact sequence \eqref{eqn: SES for Htmf} gives a multiplicative filtration of $H_*\tmf$ by $A_*$-comodules. More precisely, we filter $H_*\tmf$ by setting $F_0H_*\tmf = H_*\tmf$ and $F_1H_*\tmf:= (b_4)$, the ideal generated by $b_4$. Since $b_4$ is a comodule primitive, this is a filtration by comodules. The algebraic spectral sequence is then the spectral sequence associated to this filtration.  Since the filtration is multiplicative, the spectral sequence is as well. Moreover, there is an isomorphism of $A_*$-comodule algebras
	\[
	E_0H_*\tmf \cong \B\otimes E(b_4)
	\]
	with $\scr{B}$ concentrated in filtration degree 0 and $b_4$ a comodule primitive in filtration degree 1. Thus
	\[
	E_1^{s,t,f}\cong \Ext_\Gamma(\F_3)\otimes E(b_4)
	\]
	as a graded ring. Note that the bidegree of $b_4$ is $(0, 8, 1)$. Since this spectral sequence arises from a long exact sequence in $\Ext$, there is only a $d_1$-differential which has the form 
	\[
	d_1: E_1^{s,t,0}\to E_1^{s+1,t,1}.
	\]
	In depicting charts we will always use the Adams indexing convention and use the axes $(t-s,s)$ and suppress the filtration degree.

Below is a chart (Figure \ref{sseq: E1 and E2 pages of algebraic spectral sequence}) for the $E_1$-page of the algebraic spectral sequence. The classes in blue are those in the coset for $b_4$ in the $E_1$-page. In other words, they have filtration 1 with respect to the filtration on $H_*\tmf$. Note that the tri-degree of the $d_1$-differential implies that all differentials originate from a black class and target a blue class.

\begin{figure}
	\includegraphics[angle=90, height=\textheight]{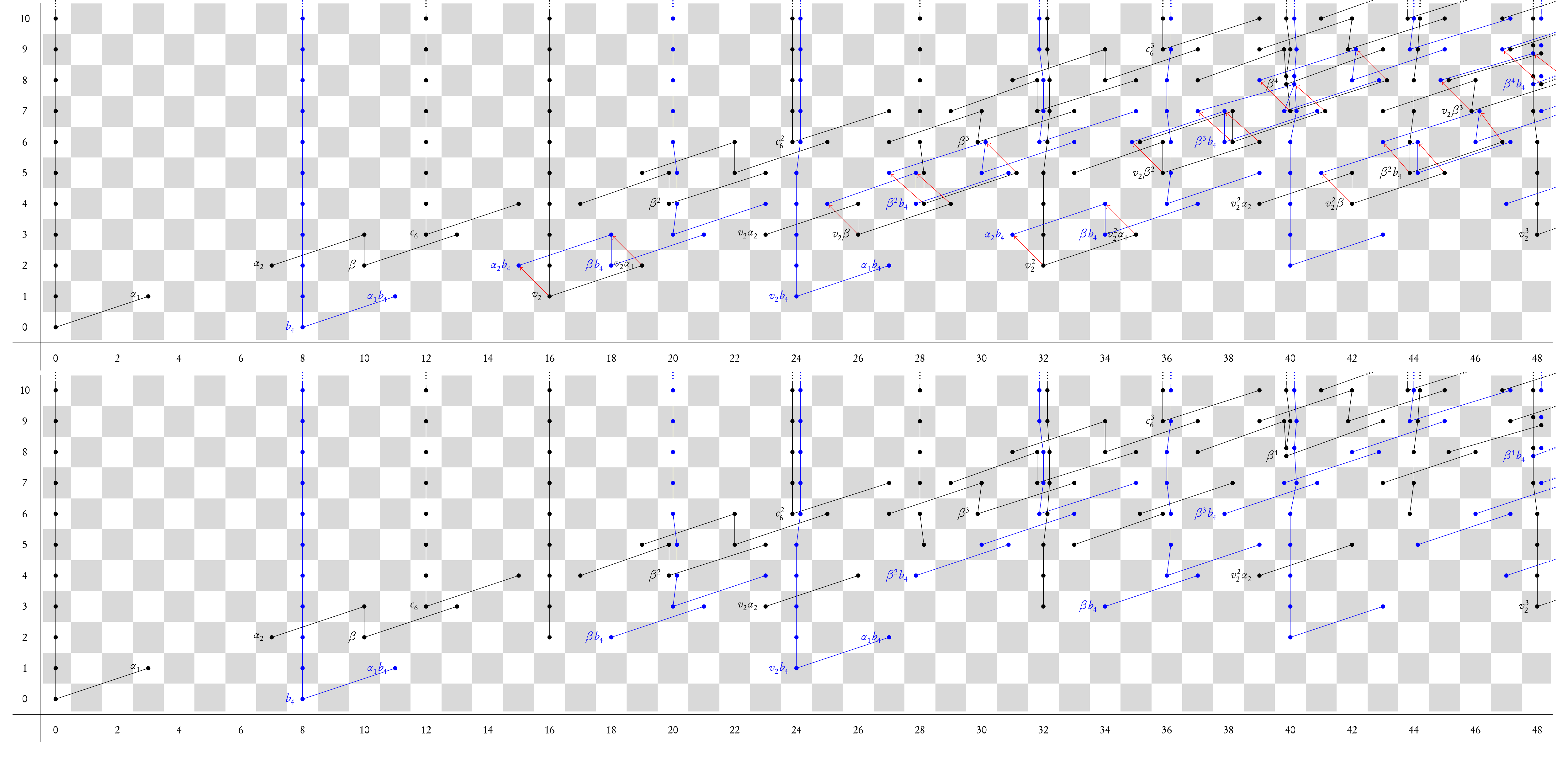}
	\caption{The $E_1$ and $E_2$-page of the algebraic spectral sequence: $
	\Ext_{\Gamma}(\F_3)\otimes E(b_4)\implies \Ext_{A_*}(H_*\tmf)
	$}\label{sseq: E1 and E2 pages of algebraic spectral sequence}
\end{figure}

We will now determine the differentials in the algebraic spectral sequence. First, we make the following simple observation. 

\begin{lem}
For degree reasons, the classes $\alpha_1, \alpha_2, b_4, \beta$, and $c_6$ are permanent cycles of the algebraic spectral sequence. 	
\end{lem}
This observation and the multiplicativity of the spectral sequence eliminate many possible differentials.

From the known computation of $\pi_*\tmf$ (cf. \cite{Bauer_2008}), we see that $\pi_{15}\tmf=0$. In the $E_1$-term of the algebraic spectral sequence, there are two classes in stem 15; the class $b_4\alpha_2$ and the class $c_6\alpha_1$. Both of these classes must die, but for degree reasons the only possibility is the following differential\footnote{The class $c_6\alpha_1$ will be dealt with by an Adams differential.}
\[
d_1(v_2)\, \dot{=}\, b_4\alpha_2.
\] 
Multiplicativity of the spectral sequence and the previous lemma yields the following result.

\begin{prop}\label{prop: algcSS differentials}
	The algebraic spectral sequence has the following $d_1$-differentials 
	\begin{align*}
		d_1(v_2^iv_0^jc_6^k\beta^\ell\alpha_1^\epsilon) &\,\dot{=}\, v_2^{i-1}v_0^jc_6^k\beta^\ell\alpha_1^\epsilon b_4\alpha_2 & i\not\equiv 0 \mod 3
	\end{align*}
	for natural numbers $i, j, k, \ell$ and $\epsilon\in \{0,1\}$. There are no other differentials.
\end{prop}

Consequently, this spectral sequence is periodic on the element $v_2^3$. 

\begin{rmk}
	It would be nice to have an argument for this differential from first principles, but the author is not currently aware of one. He suspects this implies the existence of an interesting coproduct on $H_*\tmf$.
\end{rmk}

%
%

\subsection{Algebraic $E_\infty$-term}\label{subsec: algebraic Einfty term}
We will now describe a few patterns which make up the $E_\infty$-page of the algebraic spectral sequence. We will describe these patterns as certain modules over $\Ext_{A(1)_*}(\F_3)$ along with the monomial of the algebraic spectral sequence which generates it.

\begin{enumerate}
	\item[(Pattern 1)] Since $v_2^3$ is a permanent cycle, we have the free $\Ext_{A(1)_*}(\F_3)$ modules on the powers of $v_2^3$ and the $v_2^3$-multiples of $v_2^2b_4$, i.e. for all $j\geq 0$,
		\[
			\Ext_{A(1)_*}(\F_3)\cdot \{v_2^{3j}, v_2^{3j+2}b_4\} ;
		\]
	\item[(Pattern 2)] For $j\equiv 0,1\mod 3$, we have the patterns
	\[
	\Ext_{A(1)_*}(\F_3)/(\alpha_2)\cdot \{v_2^jb_4\}
	\]
	\item[(Pattern 3)] For $j\not\equiv 0 \mod 3$, we have the following patterns 
	\[
	\Ext_{A(1)_*}(\F_3)/(\alpha_1,\alpha_2,\beta)\cdot\{v_0v_2^j\}\oplus \Ext_{A(1)_*}(\F_3)/(\alpha_2, v_0)\cdot \{v_2^j\alpha_2\}.
	\]
\end{enumerate}
The way we obtained these patterns was by noting that, as a module over $\Ext_{A(1)_*}(\F_3)$, the $E_1$-page of the algebraic spectral sequence is freely generated by the monomials $v_2^jb_4^{\epsilon}$. In other words, we have an isomorphism of $\Ext_{A(1)_*}(\F_3)$-modules,
\[
E_1\cong \bigoplus_{j\geq 0,\, \varepsilon\in \{0,1\}}\Ext_{A(1)_*}(\F_3)\cdot\{v_2^jb_4^\varepsilon\}.
\]
The three patterns arise by partitioning the free modules $\Ext_{A(1)_*}(\F_3)\cdot \{v_2^jb_4^\epsilon\}$ into those which neither receive nor support any differentials (Pattern 1), receive differentials (Pattern 2), or support differentials (Pattern 3). 

\begin{rmk}\label{rmk: pattern terminology}
In later parts of this paper we will need to refer to these patterns. We will refer to them as \emph{patterns of type $j$ on generator $x$}. So for example, if we look at the pattern on the Adams $E_2$-term generated by the monomial $v_2^4b_4$, then we will call this a pattern of type 2 on generator $v_2^4b_4$. For patterns of the third type, we will call these patterns of type 3 on generator $v_2^j$. This is potentially confusing since $v_2^j$ does not survive the algebraic spectral sequence unless $j$ is a multiple of 3. This terminology stems from the fact that this pattern is the residual piece of a free $\Ext_{A(1)_*}(\F_3)$-module generated by $v_2^j$.
\end{rmk}

\begin{figure}[ht]
	\includegraphics[width = \textwidth, keepaspectratio]{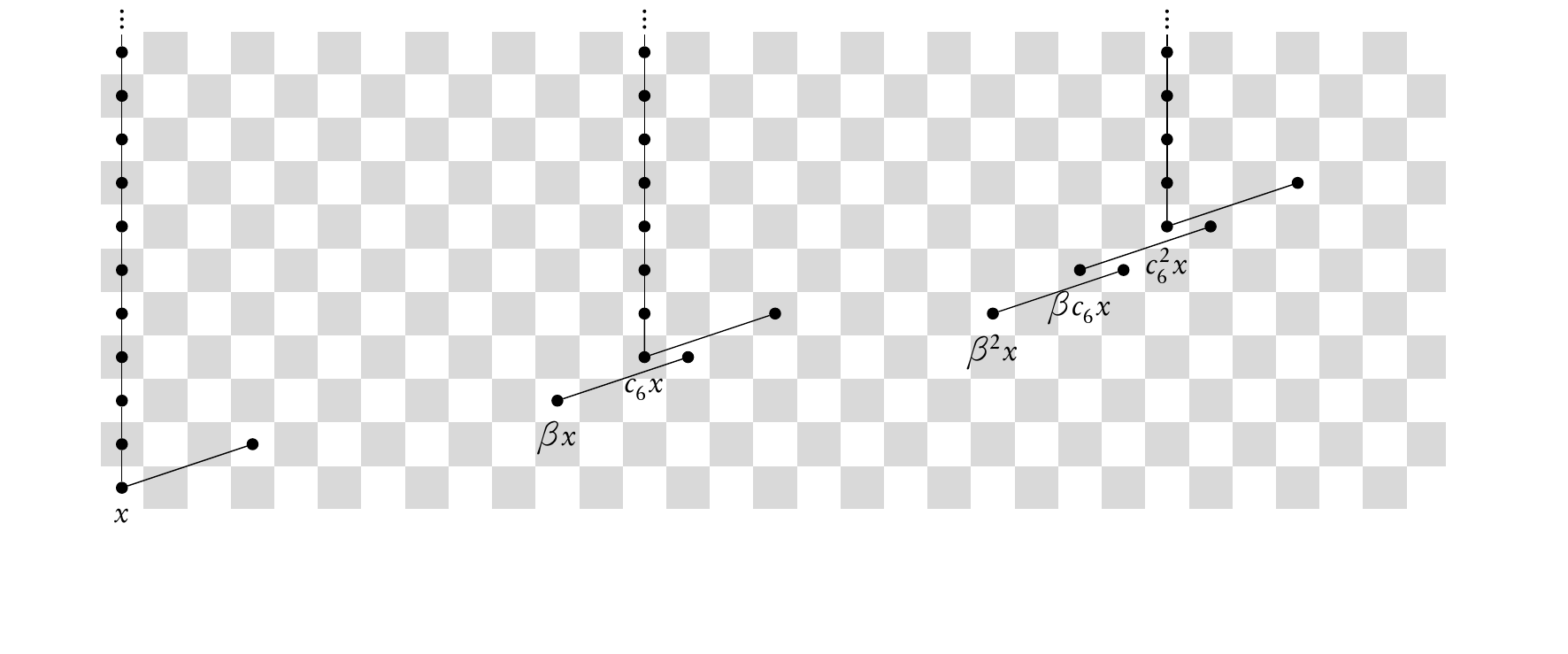}
	\caption{A depiction of pattern 2 on a generator $x$}
\end{figure}

\begin{figure}[ht]
	\includegraphics[width = \textwidth, keepaspectratio]{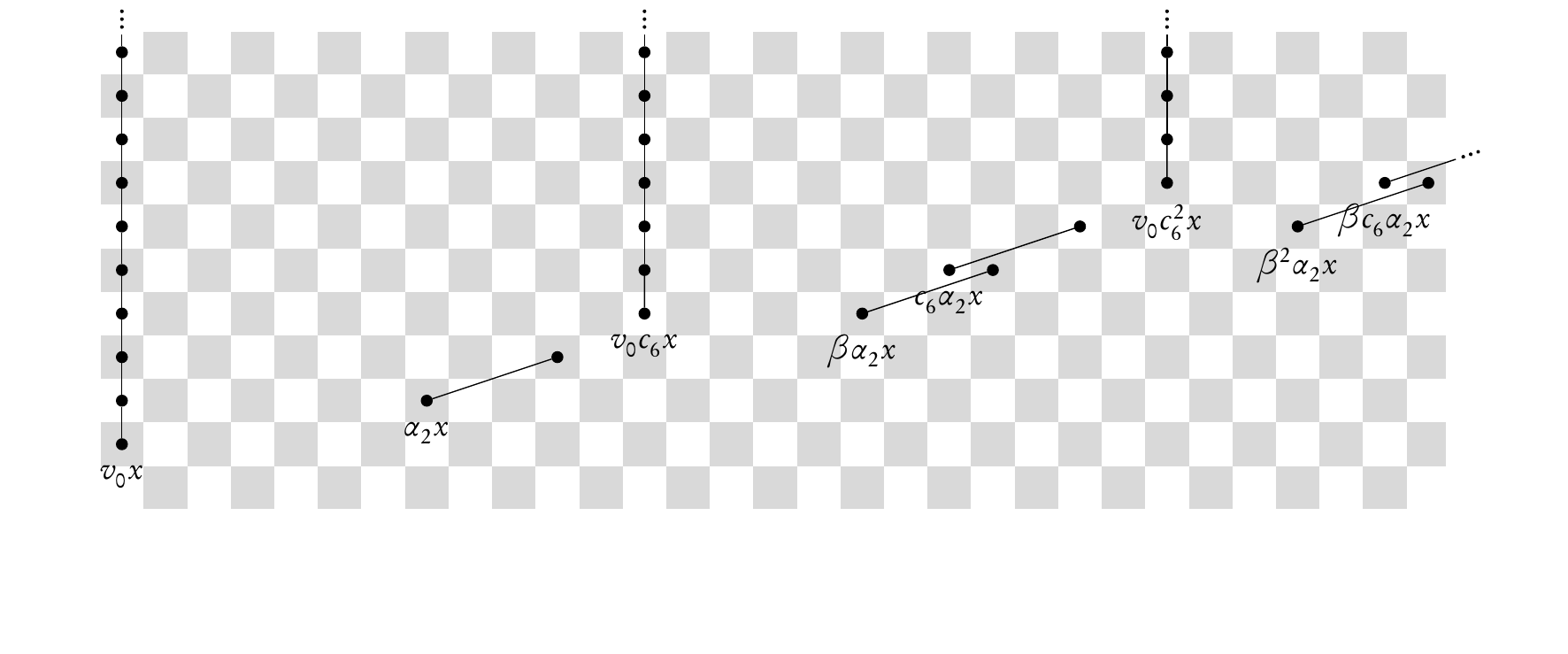}
	\caption{A depiction of pattern 3 on a generator $x$}
\end{figure}


\subsection{Algebraic hidden extensions}
As with any spectral sequence, there is the possibility of extension problems. We  will show that there is a crucial hidden $v_0$-extension which will play an important role in the next section. Namely,

\begin{prop}\label{prop: hidden extn v0(v2^2alpha2) = v2b4c6alpha1}
	In the algebraic spectral sequence, there is a hidden multiplicative extension
	\[
	v_0\cdot (v_2^2\alpha_2) \,\dot{=}\, v_2b_4c_6\alpha_1,
	\]
	consequently for every natural number $j$ and $k$, we have the hidden extension
	\[
	v_0\cdot v_2^2c_6^j\beta^k\alpha_2\,\dot{=}\, v_2c_6^{j+1}\beta^kb_4\alpha_1.
	\]
\end{prop}

\begin{rmk}
One might protest that this is not a hidden extension since $v_2b_4c_6\alpha_1$ is an element in the correct Adams filtration. However, from the perspective of the algebraic spectral sequence, $v_2^2\alpha_2$ has filtration 0 and $v_2b_4c_6\alpha_2$ has filtration 1. Since $v_0$ has filtration 0, this is in fact a hidden extension.
\end{rmk}


Before proving this, we will need to show the following. 

\begin{lem}\label{lem: Massey prod}
	In $\Ext_{A(1)_*}(\F_3)$, there is the Massey product 
	\[
	c_6\alpha_1\in \langle v_0, \alpha_2, \alpha_2\rangle.
	\]
	and there is zero indeterminacy.
\end{lem}

To prove this, we need to recall May's Convergence theorem. The proof of this fact can be found as Theorem 4.1 of \cite{MR0238929}, but we will only be interested in the case of a three-fold Massey product. The variant we will use is Theorem 2.2.2 of \cite{StableStems}. However, since we are working at an odd prime, we have to keep track of signs. In the following, if $x$ is a class in degree $n$, then we write $\overline{x}$ for $(-1)^{n+1}x$ (see \cite[Appendix 4]{greenbook} for further details).

\begin{thm}[May's Convergence Theorem]
	Let $\alpha_0, \alpha_1, \alpha_2$ be elements of $\Ext$ such that the Massey product $\langle \alpha_0, \alpha_1, \alpha_2\rangle$ is defined. For each $i$, let $a_i$ be a permanent cycle on the May $E_r$-page which detects $\alpha_i$. Suppose further that 
	\begin{enumerate}
		\item The Massey product $\langle a_0,a_1,a_2\rangle$ is defined on the $E_{r+1}$-page: there are $a_{01}$ and $a_{12}$ such that $d_r(a_{01}) = \overline{a_0}a_1$ and $d_r(a_{12}) = \overline{a_1}a_2$.
		\item If $(m,s,t)$ is the tri-degree of either $a_{01}$ or $a_{12}$, and for for any $m'\geq m$ and $q$ such that $m'-q<m-r$, the differential
			\[
			d_q: E_q^{m', s, t}\to E_q^{m'-q+1, s+1,t}
			\]
			is zero. 
	\end{enumerate}
	Then the element $\overline{a_{01}}a_3+\overline{a_0}a_{12}$ is a permanent cycle and detects an element of $\langle \alpha_0, \alpha_1, \alpha_2\rangle$.
\end{thm}

\begin{rmk}
	The second condition in May's Convergence Theorem is often expressed by saying there are no ``crossing differentials''.
\end{rmk}

We will use May's convergence theorem to give a proof for Lemma \ref{lem: Massey prod}. From the discussion of the May spectral sequence right before Proposition \ref{prop: Ext of A(1)},  $\alpha_2:= v_1\alpha_1$ is a non-zero permanent cycle. One easily shows that 
	
	\begin{lem}
		In $\Ext_{A(1)_*}(\F_3)$, there is the Massey product 
		\[
		\alpha_2 = \langle v_0, \alpha_1, \alpha_1\rangle.
		\]
	\end{lem}

\begin{proof}[Proof of Lemma \ref{lem: Massey prod}]
	Since $\alpha_2 = v_1\alpha_1$ is exterior on the May $E_1$-page, we get the following defining system $d_1(v_1^2) = \overline{v_0}\alpha_2$ and $d_1(0) = \alpha_2^2$ for the Massey product $\langle v_0, \alpha_2, \alpha_2\rangle$. This shows that on the May $E_1$-page we have $v_1^2\alpha_2 = c_6\alpha_1$ is in $\langle v_0, \alpha_2, \alpha_2\rangle$. 
	
	Since $\alpha_2^2=0$ and $v_0\alpha_2=0$ in $\Ext_{A(1)_*}(\F_3)$, the Massey product $\langle v_0, \alpha_2, \alpha_2\rangle$ is defined in $\Ext_{A(1)_*}(\F_3)$. Furthermore, there is no indeterminacy of this Massey product. So we just need to check the second condition of May's Convergence Theorem, that there are no crossing differentials. Note that $v_1^2\in E_1^{6, 2, 10}$ and that $E_2^{m, 2, 10}=0$ for all $m$. So condition (2) is satisfied here. Likewise, note that $(v_1\alpha_1)^2\in E_1^{8, 4,18}$ and that $E_1^{m, 3, 18}=0$ for all $m$. Thus there are no nonzero differentials to worry about. So condition (2) is always satisfied here as well. Thus we may apply May's Convergence Theorem to infer that 
	\[
	c_6\alpha_1\in \langle v_0, \alpha_2, \alpha_2\rangle. 
	\]
	It is easy to see that there is no indeterminacy.
\end{proof}

\begin{rmk}\label{rmk: May Convergence in alg'c SS and CESS}
Keep in mind that May's convergence theorem is actually very general (cf. the discussion preceeding \cite[A1.4.10]{greenbook}). It applies to any spectral sequence which arises from a multiplicative filtration on a DGA. In particular, it applies to the Cartan-Eilenberg SS and the algebraic SS we have used. Since the Cartan-Eilenberg SS collapses, and since the algebraic SS only has $d_1$-differentials, the May Convergence Theorem vacuously applies to these spectral sequences.
\end{rmk} 

Thus, we derive the following corollary. 

\begin{cor}\label{cor: massey prod c6alpha1}
	In $\Ext_{A_*}(H_*\tmf)$ there is the Massey 
	\[
	c_6\alpha_1 \,\dot{=}\, \langle v_0, \alpha_2, \alpha_2\rangle.
	\]
\end{cor}

We will use this corollary to derive the hidden multiplicative extension. 

\begin{proof}[Proof of Proposition \ref{prop: hidden extn v0(v2^2alpha2) = v2b4c6alpha1}]
	One can check, using the May Convergence Theorem applied to the algebraic spectral sequence, that one has the Massey product 
	\[
	\pm v_2^2\alpha_2 \in  \langle \alpha_2, \alpha_2, b_4v_2\rangle,
	\]
	and that this Massey product has no indeterminacy. Note that we do not know the sign since we only know the differential $d_1(v_2)$ up to sign. 
	
	In order to apply the May Convergence Theorem to this Massey product, we must check that it is defined on $\Ext_{A_*}(H_*\tmf)$. Note that $\alpha_2^2=0$ in this Ext group, since there is no nonzero group in the 14 stem. Furthermore, since hidden extensions must always target a class in higher filtration, and since the algebraic spectral sequence only has elements in filtration 0 and 1, it follows that there can be no hidden extension for the product of $b_4v_2$ and $\alpha_2$. Thus the Massey product is defined in $\Ext_{A_*}(H_*\tmf)$ (cf. Remark \ref{rmk: May Convergence in alg'c SS and CESS}). Using the First Juggling Theorem (cf. \cite[A1.4.6]{greenbook}), we have 
	\[
	v_0\cdot (v_2^2\alpha_2) \, \dot{=}\, v_0\langle \alpha_2, \alpha_2, b_4v_2\rangle \,\dot{=}\, \langle v_0, \alpha_2, \alpha_2\rangle b_4v_2 \,\dot{=}\, c_6\alpha_1b_4v_2
	\]
	yielding the desired extension.
\end{proof}

We will also have occasion to use the following hidden extension. 

\begin{cor}\label{cor: hidden extn v0 v_1alpha2 = b4c6alpha1}
	In the algebraic SS, there is the Massey product
	\[
	v_2\alpha_2 \,\dot{=}\, \langle \alpha_2, \alpha_2, b_4\rangle
	\]
	and consequently the hidden extension
	\[
	v_0\cdot(v_2\alpha_2)\, \dot{=}\, b_4\cdot(c_6\alpha_1).
	\]
\end{cor}
\begin{proof}
	A defining system for the Massey product $\langle \alpha_2, \alpha_2, b_4\rangle$ on the $E_2$ page is given by $d_1(0) = \alpha_2^2$ and $d_1(v_2) = \alpha_2b_4$. Observe that there is no indeterminacy. So by May's convergence theorem we have the Massey product
	\[
	v_2\alpha_2 \,\dot{=}\, \langle \alpha_2, \alpha_2,b_4\rangle.
	\]
	Since this Massey product and $\langle v_0, \alpha_2, \alpha_2\rangle$ are both strictly defined, we get from the First Juggling Theorem \cite[A1.4.6(c)]{greenbook} the following equalities
	\[
	b_4c_6\alpha_1 \,\dot{=}\, \langle v_0, \alpha_2, \alpha_2\rangle b_4 \,\dot{=}\, v_0\langle \alpha_2, \alpha_2, b_4\rangle \,\dot{=}\, v_0(v_2\alpha_2).
	\]
\end{proof}

\subsection{Comparison to the Adams spectral sequence in $\tmf$-modules}\label{subsection: comparison}

We now make a few remarks comparing the $E_2$-term of the Adams spectral sequence for $\tmf$ and the Adams spectral sequence for $\tmf$ in $\tmf$-modules as studied by Hill (\cite[Section 2]{Hill_2007}). The latter is a spectral sequence
\[
E_2^{s,t}=\Ext^{s,t}_{A^{\tmf}_*}(\F_3)\implies \pi_{t-s}\tmf.
\]
where 
\[
A_*^{\tmf} = \pi_*\left(H\wedge_{\tmf} H\right)\cong A(1)_*\otimes E(a_2)
\]
where $a_2$ is in degree 9. This class has an interesting coproduct, but this does not concern us here. What is interesting for us, however, is that in order to compute this coproduct, Hill filters $A^{\tmf}_*$ (\cite[Theorem 2.2]{Hill_2007}), resulting in an algebraic spectral sequence
\[
E_1^{s,t,*}=\Ext^{s,t}_{E_0A^{\tmf}_*}(\F_3)\implies \Ext^{s,t}_{A^\tmf_*}(\F_3).
\]
One easily derives that 
\[
E_1^{s,t,*} = \Ext^{s,t}_{E_0A^\tmf_*}(\F_3)\cong \Ext^{s,t}_{A(1)_*}(\F_3)\otimes P(\widetilde{c_4})
\]
where $\widetilde{c_4}$ is the class represented in the cobar complex of $E_0A^{\tmf}_*$ by $[a_2]$. In particular, $\widetilde{c_4}\in E_1^{1, 9,1}$. The term $\Ext_{A(1)_*}(\F_3)$ is concentrated in filtration degree 0. It turns out that this $E_1$-page is in bijection with the $E_1$-term of our algebraic spectral sequence, but with various elements in ours in the ``wrong'' filtration. For example, our element $b_4$ corresponds to Hill's element $\widetilde{c}_4$. The former is in tridegree $(0,8,1)$ but the latter is in tridegree $(1,8,1)$.

 We provide a short dictionary relating various names in our spectral sequence to Hill's algebraic spectral sequence. 

\begin{center}
\begin{tabular}[c]{c|c}
	alg'c SS & Hill's alg'c SS\\ \hline 
	$b_4$ & $\widetilde{c_4}$\\ \hline 
	$v_2$ & $\widetilde{c_4}^2$\\ \hline
	$b_4v_2$ & $\widetilde{c_4}^3$\\ \hline
	\vdots & \vdots 
\end{tabular}
\end{center}
In particular, the element that Hill calls $\widetilde{c_4}^{2\ell+\varepsilon}$ corresponds to the element we call $b_4^\varepsilon v_2^{\ell}$. Moreover, Hill is able to derive a differential $d_1(\widetilde{c_4}) = \alpha_2$. Algebraic manipulation then yields the following differentials $d_2(\alpha_2\widetilde{c_4}^2) = v_1^3\beta$ and $d_2(v_0\widetilde{c_4}^2)=v_1^3\alpha_1$. These all correspond to various differentials we encounter in this paper as well, but interestingly, not all of them are algebraic differentials. On the one hand, the differential $d_1(\widetilde{c_4}^2) = -\widetilde{c_4}\alpha_2$ corresponds to our algebraic differential $d_1(v_2) = b_4\alpha_2$. But the differential $d_1(\widetilde{c_4}) = \alpha_2$ corresponds to an \emph{Adams} $d_2$-differential $d_2(b_4) = \alpha_2$.

In particular, half of Hill's algebraic differentials are seen in our algebraic spectral sequence, but the other half arise as Adams differentials. It is this discrepancy that makes the $\tmf$-relative Adams spectral sequence more computable as opposed to the absolute Adams spectral sequence.

%% file: rational.tex

\section{Rational Homotopy of $\tmf$ and modular forms}\label{sec: rational homotopy of tmf}

In this section, we examine the $v_0$-inverted Adams spectral sequence for $\tmf$. In this case, the $v_0$-inverted Adams spectral sequence converges to the rational homotopy groups of $\tmf$, 
\[
v_0^{-1}\Ext_{A_*}(H_*\tmf)\implies \pi_*\tmf_\Q.
\]

To determine the $v_0$-localized Adams $E_2$-term, we could take the decomposition from \textsection\ref{subsec: algebraic Einfty term} and invert $v_0$. Alternatively, we could use the $v_0$-localized algebraic spectral sequence, 
\[
v_0^{-1}\Ext_{A(1)_*}(\F_3)\otimes E(b_4)\implies v_0^{-1}\Ext_{A_*}(H_*\tmf).
\]
The latter approach is more convenient. Note that 
\[
v_0^{-1}\Ext_{\Gamma}(\F_3)\cong P(v_0^{\pm 1}, v_1^3, v_2).
\]
This shows that the $v_0$-localized algebraic $E_1$-term is concentrated in even stems, and hence collapses at $E_1$. Thus we find that
\[
v_0^{-1}\Ext_{A_*}(H_*\tmf)\cong P(v_0^{\pm 1},v_1^3,v_2)\otimes E(b_4)
\]
and it follows immediately that the $v_0$-localized ASS for $\tmf$ collapses at $E_2$. 

The spectrum $\tmf$ has a close connection to classical modular forms. The ring of integral modular forms, $MF_*$, has been known for quite some time.

\begin{thm}[cf. \cite{CourbesElliptiqueFormulaire}]
	The ring of integral modular forms is given by 
	\[
	MF_* = \Z[c_4, c_6, \Delta]/(c_4^3 - c_6^2-12^3\Delta).
	\]
\end{thm}

Here, $c_4$ and $c_6$ are the normalized Eisenstein series of weight 4 and 6 respectively. Topologically, these have degree 8 and 12 respectively. The modular form $\Delta$ is often referred to as the \emph{modular discriminant}, and is a modular form of weight 12. The precise relationship between $\pi_*\tmf$ and integral modular forms is made by examining the Adams-Novikov spectral sequence. We need to make use of the following.  

\begin{thm}[cf. \cite{Bauer_2008}, \cite{Henriques}]\label{thm: edge homomorphism}
	The edge homomorphism for the Adams-Novikov spectral sequence for $\tmf$ is a map 
	\[
	\pi_*\tmf\to MF_*(\Z_{(3)}):= MF_*\otimes\Z_{(3)}
	\]
	where $MF_*$ is the ring of classical integral modular forms. The image of this map contains $c_4, c_6, 3\Delta, 3\Delta^2$ and $\Delta^3$. Moreover, this map is a rational isomorphism. 
\end{thm}

\begin{rmk}
	Since the edge homomorphism is a map of rings, the theorem determines the image entirely. 
\end{rmk}

\begin{cor}\label{cor: rational homotopy of tmf}
	The rational homotopy groups of $\tmf$ are given by 
	\[
	\pi_*\tmf_\Q\cong \Q[c_4, c_6].
	\]
\end{cor}

Theorem \ref{thm: edge homomorphism} will allows us to determine what some of the elements in the Adams $E_\infty$-term detect in $\pi_*\tmf$. It will also be used to later to establish hidden multiplicative extensions in \textsection \ref{subsec: hidden extensions}. We can carry some of this out even now. 

\begin{prop}\label{prop: v_1^3 detects c6}
	The class $v_1^3$ in the Adams $E_2$-term for $\tmf$ survives to a non-zero element in $E_\infty$ and detects the class $c_6$. 
\end{prop}
\begin{proof}
	It follows for degree reasons that $v_1^3$ cannot support or be targeted by a differential. This implies that $v_1^3$ survives to a nonzero element in $E_\infty$. From Theorem \ref{thm: edge homomorphism}, we know that some torsion free element in the 12 stem must detect $c_6$. From the Adams spectral sequence calculation we see that the only $v_0$-tower in stem 12 is the one generated by $v_1^3$. Thus $v_1^3$ detects $c_6$. 
\end{proof}

%% file: AdamsDiffs.tex

\section{Adams differentials}\label{sec: differentials}

In this section, we will determine the differentials in the Adams spectral sequence for $\tmf$. Since $\tmf$ is a commutative ring spectrum, the Adams spectral sequence is multiplicative. Begin by noting that there are several classes which are permanent cycles for degree reasons. 

\begin{lem}
The classes $v_0,\alpha_1, \alpha_2, \beta$, and $c_6$ are all permanent cycles for the Adams spectral sequence. Consequently, the differentials in the Adams spectral sequence are linear over $\Ext_{A(1)_*}(\F_3)$.
\end{lem}
This observation is very useful for our calculation for the following reason. In the last section, we have expressed the Adams $E_2$-term as a direct sum of certain patterns which were modules over $\Ext_{A(1)_*}(\F_3)$. This observation implies that the only nonzero $d_2$-differentials in the Adams spectral sequence will originate on the monomials which generate these patterns. We will also make frequent use of the following facts about $\pi_*\tmf$.  

\begin{thm}[cf. \cite{Henriques}]
	The homotopy groups of $\tmf$ are 72 periodic. Furthermore, the torsion in $\pi_*\tmf$ is concentrated in stems 3, 10, 13, 20, 27, 30, 37, and 40 modulo 72.
\end{thm}

We will begin by determining all of the length 2 differentials.

\subsection{Adams $d_2$-differentials}

As was mentioned previously, the Adams differentials are all linear over $\Ext_{A(1)_*}(\F_3)$, which means we only have to figure which of the following families of monomials support Adams $d_2$-differentials: For any natural number $j$
\begin{enumerate}
	\item $v_2^{3j}$,
	\item $v_2^jb_4$,
	\item $v_0v_2^j$ for $j\equiv 1, 2\mod 3$, and 
	\item $v_2^j\alpha_2$ for $j\equiv 1, 2\mod 3$.
\end{enumerate}
From our charts (Figures \ref{fig: Adams E_2 term 0-40} and \ref{fig: Adams E_2 term 40-80}  below), one sees that $v_2^3$ can support a length 3 differential at minimum. Thus, $v_2^{3j}$ is a $d_2$-cycle for all $j$. Moreover, there are several multiplicative relations on the Adams $E_2$-term which we get from the previous section. For example, we have $v_2^3\cdot(v_2b_4) = v_2^4b_4$.  Consequently, we have 

\begin{prop}
	The Adams $d_2$-differentials for $\tmf$ are linear over $\Ext_{A(1)_*}(\F_3)\otimes P(v_2^3)$. Thus, we only need to determine which of the monomials $b_4$, $v_0v_2$, $v_2\alpha_2$, $b_4v_2$, $v_0v_2^2$, $v_2^2\alpha_2$, and $b_4v_2^2$ support $d_2$-differentials.
\end{prop}
%

\begin{prop}\label{prop: d2b4}
There is an Adams $d_2$-differential
\begin{equation}\label{eq: d2b4}
d_2(b_4)\,\dot{=}\, \alpha_2.
\end{equation}
\end{prop}
\begin{proof}[Proof 1]
From the known computation of $\pi_*\tmf$, it is seen that $\pi_7\tmf=0$. Thus, the class $\alpha_2$ must die. The only possibility is the claimed differential.
\end{proof}
\begin{proof}[Proof 2]
	Recall that in the Adams $E_2$-term for $S^0$, we have the Massey product
	\[
	\beta_1 = \langle \alpha_1, \alpha_1, \alpha_1\rangle. 
	\]
	In the homotopy groups for the sphere, there is the same Toda bracket by Moss' convergence theorem (see \cite{Moss_1970} or \cite[Theorem 3.1.1]{StableStems}). Note also that $\alpha_1$ is represented in the cobar complex by $[\zeta_1]$. 
	
	In the Adams $E_2$-term for $\tmf$, there are classes of the same name and the same Massey product holds. In the cobar complex for $\tmf$, the class $\alpha_1$ is also represented by $[\zeta_1]$. Thus, under the induced map on $\Ext$ groups
	\[
	\Ext_{A_*}(\F_3)\to \Ext_{A_*}(H_*\tmf)
	\]
	$\alpha_1$ is sent to $\alpha_1$. The Massey product then shows that $\beta_1$ is mapped to $\beta$. Hence, $\alpha_1$ and $\beta$ are in the Hurewicz image of the sphere. 

It is known from the Adams spectral sequence for the sphere that there is a $d_2$-differential whose target is $v_0\beta$ (cf. \cite[Figure 1.2.15]{greenbook}). Since $\beta$ is in the Hurewicz image, it follows that $v_0\beta=0$ in $\pi_*\tmf$. This forces the Adams differential
	\[
	d_2(b_4\alpha_1)\, \dot{=}\, v_0\beta = \alpha_1\alpha_2.
	\]
	However, as $\alpha_1$ is a permanent cycle for the ASS for $\tmf$, we must have that 
	\[
	d_2(b_4)\, \dot{=}\, \alpha_2
	\]
	as stated.
\end{proof}

\begin{rmk}
	This is one of the Adams differentials which occurs as an algebraic differential in \cite{Hill_2007}.
\end{rmk}

We can draw an interesting consequence from the second argument provided above (we learned this from Mike Hill and Mark Behrens). 

\begin{cor}
	The element $b_4\alpha_1$ is the image of $h_1\in \Ext(S^0)$ under the map 
	\[
	\Ext(S^0)\to \Ext(\tmf),
	\]
	and consequently we have the hidden comodule extension in $H_*\tmf$, 
	\[
	\alpha(\zeta_1^3) = \zeta_1^3\otimes 1- \zeta_1\otimes b_4+ 1\otimes \zeta_1^3
	\]
	where $\alpha: H_*\tmf\to A_*\otimes H_*\tmf$ denotes the $A_*$-coaction on $H_*\tmf$. 
\end{cor}
\begin{proof}
	In the Adams spectral sequence for the sphere, it is the class $h_1$ which supports a $d_2$-differential killing $v_0\beta$. Naturality of the Adams spectral sequence implies that $h_1$ maps to $b_4\alpha_1$. 
	
	In the cobar complex for $S^0$, the element $h_1$ is represented by $\zeta_1^3$. On the other hand, we can represent $b_4\alpha_1$ in the cobar complex for $\tmf$ by $[\zeta_1]b_4$. Thus there must be an element of $H_*\tmf$ which bounds the difference between $[\zeta_1^3]$ and $[\zeta_1]b_4$. The only possibility is 
	\[
	d([]\zeta_1^3) = [\zeta_1^3]-[\zeta_1]b_4.
	\]
	This implies the claimed coaction. 
\end{proof}

\begin{prop}
	There is an Adams $d_2$-differential 
	\begin{equation}\label{eq: d2v0v2}
	d_2(v_0v_2)\,\dot{=}\,c_6\alpha_1.
	\end{equation}
\end{prop}
\begin{proof}[Proof 1]
	It is known that $\pi_{15}(\tmf)=0$. The only nonzero class in this stem on the Adams $E_2$-term for $\tmf$ is $c_6\alpha_1$. Thus, this class must die. The only possibility is the claimed differential.
\end{proof}

\begin{proof}[Proof 2]
	We provide a second proof which does not rely on a priori knowledge of $\pi_*\tmf$. Recall the Massey product for $c_6\alpha_1$ we found in Corollary \ref{cor: massey prod c6alpha1}. Since $\alpha_2$ projects to 0 on the $E_3$-page, we have that the Massey product projects to 0 at $E_3$. One also checks that the indeterminacy for this Massey product on $E_3$ is 0. It is also the case that there is no room for crossing differentials in this range. Thus Moss' Convergence Theorem (\cite{Moss_1970}, \cite[Theorem 3.1.1]{StableStems}) implies that $c_6\alpha_1$ must project to 0 in $E_\infty$. This implies that $c_6\alpha_1$ must be killed by a $d_2$-differential. The only possibility is the claimed differential.
\end{proof}

One might be tempted to conclude from this that there is a length 2 differential from $b_4v_2$ to $v_2\alpha_2$. However, one must be cautious. Even though $b_4v_2$ was a product in the $E_1$-term of the algebraic spectral sequence of the last section, it is no longer decomposable (as $v_2$ supported an algebraic differential). In fact, this differential does not occur. As explained in subsection \ref{subsection: comparison}, the classes $b_4$ and $v_2$ correspond to Hill's classes $\widetilde{c_4}$ and $\widetilde{c_4}^2$ respectively. Also, $b_4v_2$ corresponds to $\Delta$, the modular discriminant. In any of the computations for $\pi_*\tmf$, there is a differential $d_?(\Delta) = \alpha_1\beta^2$. This suggests that $b_4v_2$ ought to support a length 3 differential to $\alpha_1\beta^2$. On the other hand, $\pi_{23}\tmf=0$, and on $E_2(\tmf)$, there are the nonzero classes $v_2\alpha_2, \alpha_1\beta^2$, and $b_4c_6\alpha_1$. Also, Proposition \ref{prop: d2b4} implies that $d_2(c_6b_4\alpha_1) = c_6\alpha_1\alpha_2$, taking care of the class $c_6b_4\alpha_1$. This suggests that $v_2\alpha_2$ will support a differential. 

\begin{prop}
	In $\pi_*\tmf$, one has that $c_6\beta=0$. Consequently, there is an Adams $d_2$-differential 
	\begin{equation}\label{eq:d2v2alpha2}
	d_2(v_2\alpha_2)\,\dot{=}\, c_6\beta.
	\end{equation}
\end{prop}
We give two proofs. 
\begin{proof}[Proof 1]
	By the previous proposition, we can form the Massey product $\langle c_6, \alpha_1, \alpha_1\rangle_{E_3}$ on the $E_3$-page. By the juggling lemma, \cite[Appendix 1]{greenbook}, we have that 
	\[
	c_6\beta = c_6\langle \alpha_1, \alpha_1,\alpha_1\rangle = \langle c_6, \alpha_1,\alpha_1\rangle \alpha_1.
	\]
	From the previous proposition, we infer that the Massey product $\langle c_6, \alpha_1, \alpha_1\rangle$ contains 0. It is also easy to see that this Massey product has zero indeterminacy. Thus $c_6\beta=0$ in $E_3(\tmf)$. Thus $c_6\beta$ must be the target of a $d_2$-differential. The only possible source is $v_2\alpha_2$. 
\end{proof}

\begin{proof}[Proof 2]
	From Proposition \ref{prop: d2b4}, we deduce that 
	\[
	d_2(b_4c_6\alpha_1) = c_6\alpha_1\alpha_2 = v_0c_6\beta.
	\]
	The hidden extension \ref{cor: hidden extn v0 v_1alpha2 = b4c6alpha1} then implies the stated differential.
\end{proof}

The next monomials we need to consider are $b_4v_2, v_0v_2^2$, and $v_2^2\alpha_2$, in that order. By inspection of the chart, each of these classes have only one possible target on the $E_2$-page. However, one finds from the previous propositions that each of these potential targets actually supports a differential. Thus $b_4v_2, v_0v_2^2$, and $v_2^2\alpha_2$ are $d_2$-cycles. Thus we move on to the  monomial $b_4v_2^2$.

\begin{prop}
	There is a $d_2$-differential 
	\begin{equation}\label{eq: d2b4v2^2}
	d_2(b_4v_2^2)\,\dot{=}\, v_2^2\alpha_2.
	\end{equation}
	as well as the $d_2$-differential 
	\begin{equation}\label{eq:d2v0b4v2^2}
	d_2(v_0b_4v_2^2)\, \dot{=}\, v_2c_6b_4\alpha_1 
	\end{equation}
\end{prop}
\begin{proof}[Proof 1]
	It is known that $\pi_{39}\tmf$ is zero (cf. \cite{Bauer_2008, supplementary}), but on the $E_2$-term, there are the nonzero classes $v_2^2\alpha_2$ and $v_2c_6b_4\alpha_1$ which are not killed by previously established $d_2$-differentials. The only way for $v_2^2\alpha_2$ to be killed is by a $d_2$-differential given by the claimed differential. The hidden $v_0$-extension established in Proposition \ref{prop: hidden extn v0(v2^2alpha2) = v2b4c6alpha1} gives us the second differential.
\end{proof}

\begin{proof}[Proof 2]
	We have already established the differential $d_2(v_0v_2) = c_6\alpha_1$. Since $v_2b_4$ is a $d_2$-cycle, we have that 
	\[
	d_2((v_0v_2)v_2b_4) = v_2b_4c_6\alpha_1. 
	\]
	However, in the algebraic spectral sequence, we had the relation 
	\[
	(v_0v_2)v_2b_4 = v_0(b_4v_2^2).
	\]
	The multiplicativity of the spectral sequence and the hidden extension in Proposition \ref{prop: hidden extn v0(v2^2alpha2) = v2b4c6alpha1} implies the differential $d_2(b_4v_2^2) = v_2^2\alpha_2$.
\end{proof}

We can draw from this differential another $d_2$-differential. 

\begin{cor}
	There is the following $d_2$-differential
	\begin{equation}\label{eq: d2v2^2b4alpha2}
	d_2(v_2^2b_4\alpha_2)\, \dot{=}\, v_2c_6b_4\beta.
	\end{equation}
\end{cor}
\begin{proof}
	From the previous proposition we deduce the differential 
	\[
	d_2(v_2^2b_4\beta) \,\dot{=}\, v_2^2\alpha_2\beta.
	\]
	However, we have from Proposition \ref{prop: hidden extn v0(v2^2alpha2) = v2b4c6alpha1} that 
	\[
	v_0(v_2^2\alpha_2\beta)\,\dot{=}\, c_6b_4v_2\beta\alpha_1.
	\]
	This implies the differential
	\[
	d_2(v_2^2v_0b_4\beta)\,\dot{=}\, c_6b_4v_2\beta\alpha_1.
	\]
	However, since $v_0\beta = \alpha_1\alpha_2$, we also have 
	\[
	v_2^2v_0b_4\beta = (v_2^2b_4\alpha_2)\cdot \alpha_1.
	\]
	Since $\alpha_1$ is a permanent cycle, multiplicativity of the spectral sequence implies the claimed differential.
\end{proof}

This completes the determination of the Adams $d_2$-differential. Below, in Figure \ref{fig: Adams E_2 term 0-40} and  \ref{fig: Adams E_2 term 40-80}, we depict that Adams $E_2$-term along with the $d_2$-differentials. The reader will notice that we have used several different colors in the chart. Here is a key to the use of these colors.

\begin{center}
\begin{tabular}[c]{c|c|c}
	Color & Pattern & generator\\ \hline 
	black & pattern 1& 1\\ \hline 
	\textcolor{blue}{blue} & pattern 2 & $b_4$ \\ \hline
	\textcolor{limegreen}{lime green} & pattern 3 & $v_2$\\ \hline
	\textcolor{darkmagenta}{dark magenta} & pattern 2 & $v_2b_4$ \\ \hline
	\textcolor{lavenderrose}{lavender rose} & pattern 3& $v_2^2$\\ \hline
	\textcolor{teal}{teal} & pattern 1 & $v_2^2b_4$\\ \hline
	\textcolor{seagreen}{sea green} & pattern 1 & $v_2^3$ \\ \hline
	\textcolor{maroon}{maroon} & pattern 2 & $v_2^3b_4$ \\ \hline
	\textcolor{darkviolet}{dark violet} & pattern 3 & $v_2^4$\\ \hline
	\textcolor{twelve}{blue gray} & pattern 2 & $v_2^4b_4$
\end{tabular}.
\end{center}

Before proceeding onto computing the $d_3$-differentials, we will give a description of the Adams $E_3$-term based on the differentials we just found. 

\begin{figure}
\centering
\begin{subfigure}{0.49\textwidth}
	\includegraphics[angle=90,height=\textheight]{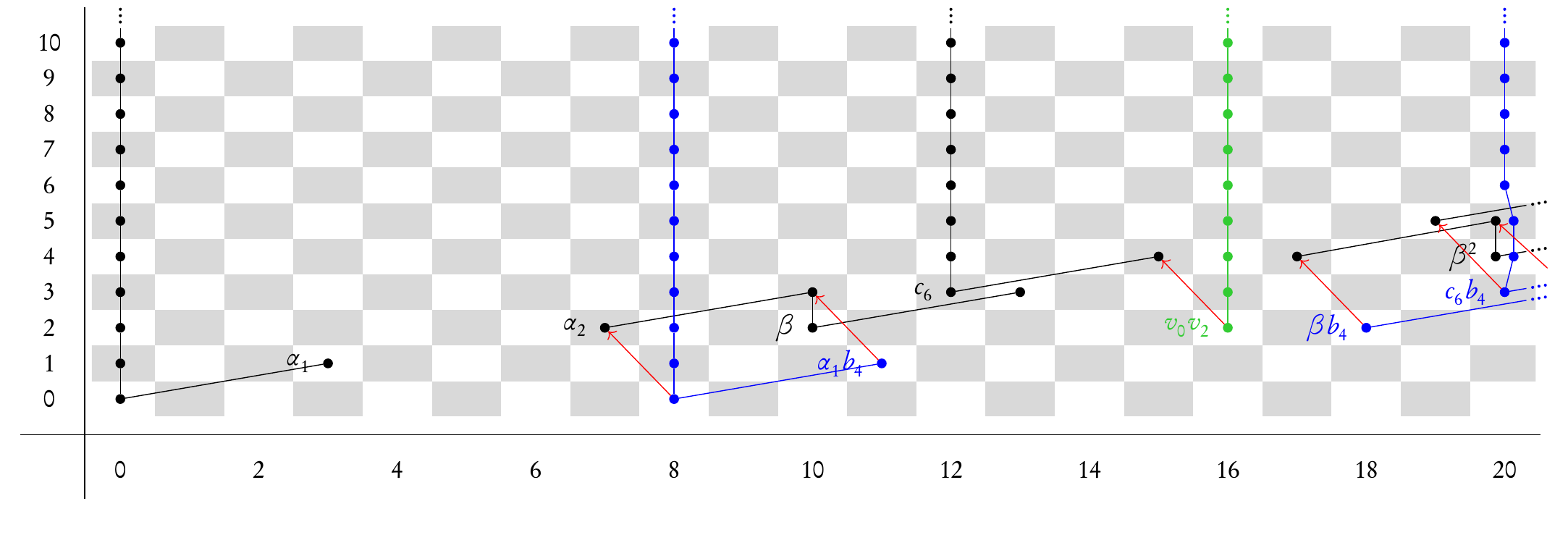}
	\caption{}
\end{subfigure}
\begin{subfigure}{0.49\textwidth}
	\includegraphics[angle=90,height=\textheight]{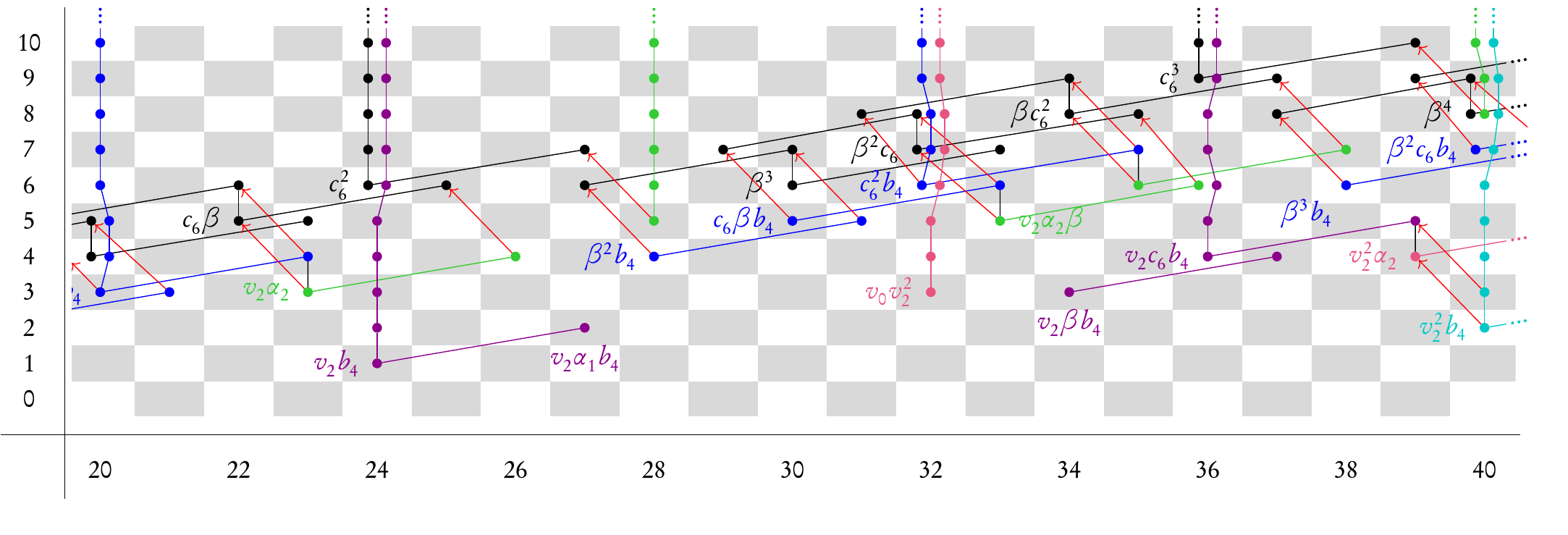}
	\caption{}
\end{subfigure}
\caption{Adams $E_2$-page in stems 0-40 with $d_2$-differentials} \label{fig: Adams E_2 term 0-40}
\end{figure}

\begin{figure}
\centering
\begin{subfigure}{0.49\textwidth}
	\includegraphics[angle=90,height=\textheight]{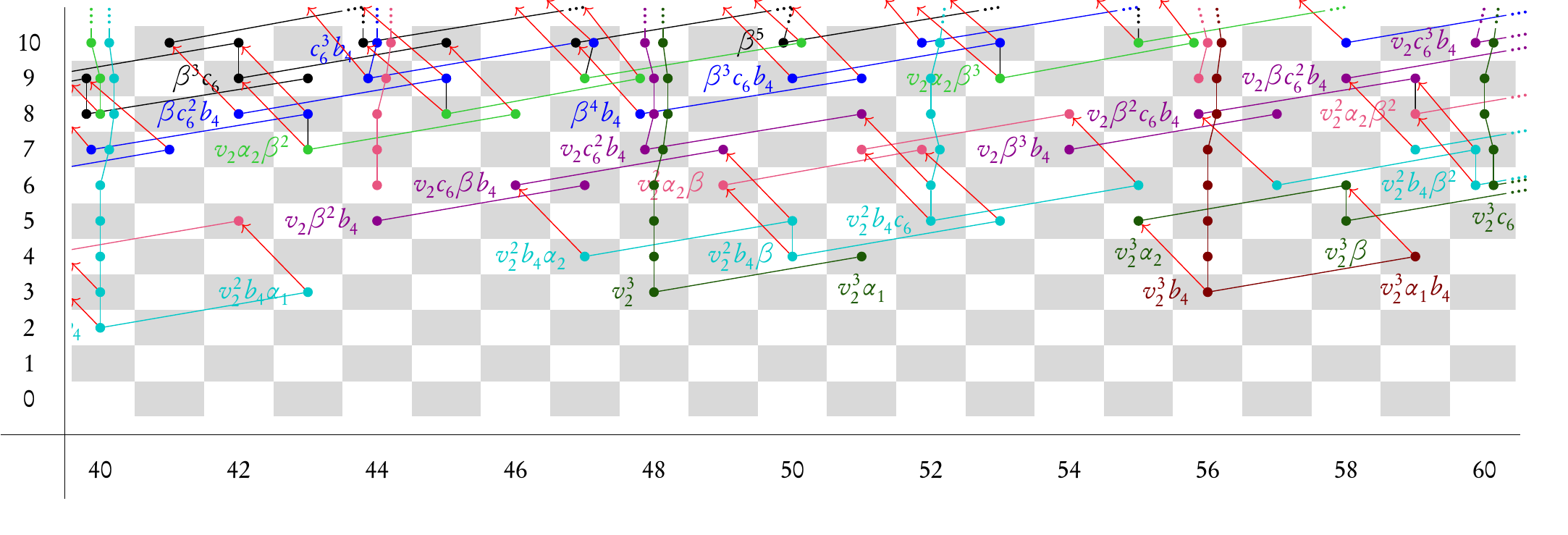}
	\caption{}
\end{subfigure}
\begin{subfigure}{0.49\textwidth}
	\includegraphics[angle=90,height=\textheight]{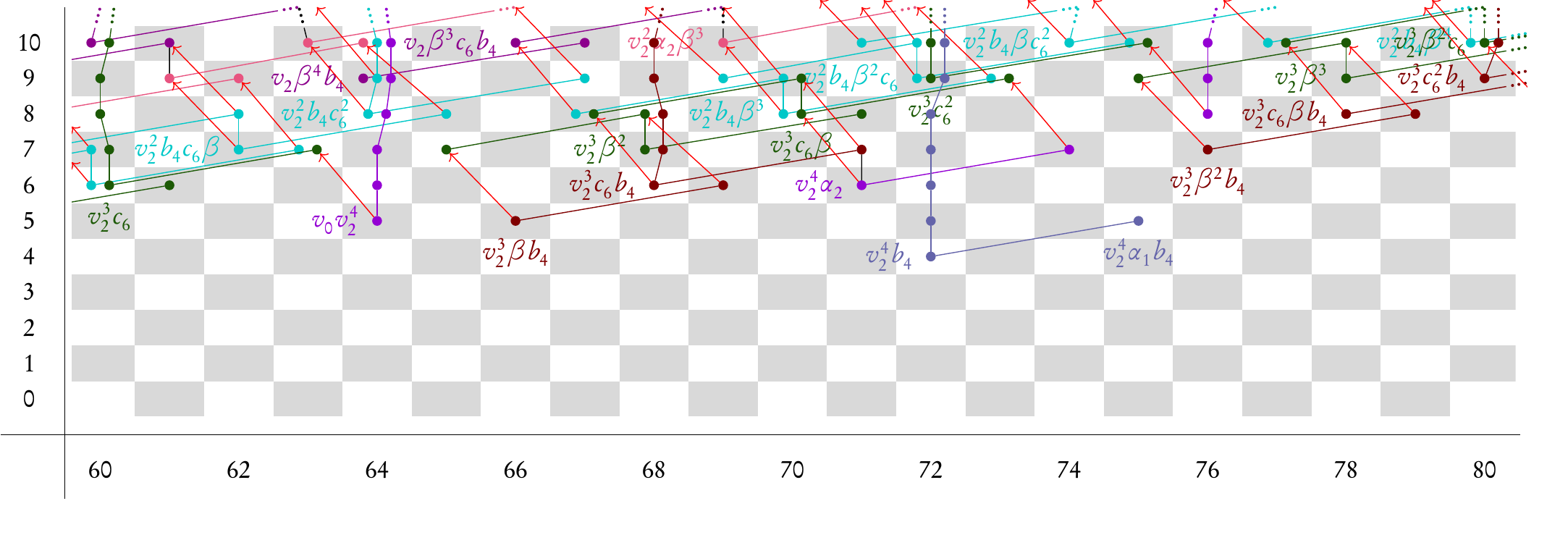}
	\caption{}
\end{subfigure}
\caption{Adams $E_2$-page in stems 40-80 with $d_2$-differentials}\label{fig: Adams E_2 term 40-80}
\end{figure}

\subsection{Determining the Adams $E_3$-term}
We now set out to determine the patterns that make up the Adams $E_3$-term for $\tmf$. To get things going, first note that the pattern of type 2 on generator $b_4$ and the pattern of type 3 on generator $v_2$ support differentials into the pattern $\Ext_{A(1)_*}(\F_3)\cdot\{1\}$ (see Remark \ref{rmk: pattern terminology} for an explanation of this terminology). More specifically, the differential \eqref{eq: d2b4} propagates to give the following $d_2$-differentials for $k, j, \ell\in \mathbb{N}$ and $\varepsilon_1\in \{0,1\}$;
\[
d_2(v_0^\ell c_6^j \alpha_1^{\varepsilon_1}\beta^kb_4)\, \dot{=}\,
\begin{cases}
	 c_6^j\alpha_1^{\varepsilon_1}\alpha_2\beta^k & \ell=0\\
	 0 & \ell\neq 0
\end{cases}.
\]
Similarly, the differentials \eqref{eq: d2v0v2} and \eqref{eq:d2v2alpha2} respectively propagate to give the differentials
\[
d_2(c_6^jv_0^\ell\cdot (v_0v_2))\, \dot{=}\, 
\begin{cases}
	c_6^{j+1}\alpha_1 & \ell=0 \\
	0 & \ell\neq 0
\end{cases}
\]
and 
\[
d_2(c_6^j\beta^k\alpha_1^{\varepsilon_1}\cdot(v_2\alpha_2))\,\dot{=}\, c_6^{j+1}\beta^{k+1}\alpha_1^{\varepsilon_1}.
\]
Observe that any monomial in $\Ext_{A(1)_*}(\F_3)$ involving an $\alpha_2$ or a $c_6\alpha_1$ is hit by a differential. So from $\Ext_{A(1)_*}(\F_3)$ we obtain the module 
\[
\Ext_{A(1)_*}(\F_3)/(\alpha_2, c_6\beta, c_6\alpha_1)\cdot\{1\}.
\]
The patterns supported by $b_4$ and $v_2$ do not receive any Adams $d_2$-differentials, so all we must do is determine what remains of these patterns after applying the Adams $d_2$-differentials. It follows from these differentials that what remains of the pattern on $b_4$ is the submodule 
\[
\Ext_{A(1)_*}(\F_3)/(\alpha_1, \alpha_2, \beta)\cdot\{v_0b_4\}
\]
and what remains of the pattern on $v_2$ is 
\[
\Ext_{A(1)_*}(\F_3)/(\alpha_1,\alpha_2,\beta)\cdot\{v_0^2v_2\}.
\]

The next pattern we need to consider is the pattern of type 2 on $b_4v_2$. Since $b_4v_2$ is a $d_2$-cycle, this entire pattern consists of $d_2$-cycles. Because of the hidden $v_0$-extension in Prop \ref{prop: hidden extn v0(v2^2alpha2) = v2b4c6alpha1}, we will consider this pattern in tandem with the half of the pattern of type 3 on $v_2^2$ generated by $v_2^2\alpha_2$; namely the module $\Ext_{A(1)_*}(\F_3)/(\alpha_2, v_0)\cdot \{v_2^2\alpha_2\}$. This half also consists only of $d_2$-cycles. Thus, the combined pattern only receives differentials. It receives its differentials from the free pattern $\Ext_{A(1)_*}(\F_3)\cdot\{v_2^2b_4\}$. The differentials \eqref{eq: d2b4v2^2}, \eqref{eq:d2v0b4v2^2}, and \eqref{eq: d2v2^2b4alpha2} propagate to give the following differentials:
\begin{align*}
d_2(v_0^jc_6^k\beta^\ell\alpha_1^{\epsilon_1}\cdot(v_2^2b_4))\, &\dot{=}\, 
	\begin{cases}
		c_6^k\beta^\ell\alpha_1^{\varepsilon_1}\cdot(v_2^2\alpha_2) & j=0\\
		c_6^{k+1}\beta^\ell\alpha_1\cdot(v_2b_4) & j=1, \varepsilon_1=0\\
		0 & else
	\end{cases}\\
d_2(c_6^k\beta^\ell(b_4v_2^2)\cdot \alpha_2) \, &\dot{=}\, c_6^{k+1}\beta^{\ell+1}(v_2b_4).
\end{align*}
Note that any monomial in the pattern $\Ext_{A(1)_*}(\F_3)/(\alpha_2)\cdot\{b_4v_2\}$ which contains a $c_6\alpha_1$ or $c_6\beta$ is hit by a differential. Also, the piece $\Ext_{A(1)_*}(\F_3)/(v_0, \alpha_2)\cdot \{v_2^2\alpha_2\}$ is annihilated by these differentials. Hence, this pattern yields the following module 
\[
\Ext_{A(1)_*}(\F_3)/(\alpha_2, c_6\beta, c_6\alpha_1)\cdot\{b_4v_2\}. 
\]
It also follows that what remains of the pattern on $b_4v_2^2$ is the submodule 
\[
\Ext_{A(1)_*}/(\alpha_1, \alpha_2, \beta)\cdot\{v_0^2v_2^2b_4\}.
\]
The other half of pattern 3 on $v_2^2$, i.e. $\Ext_{A(1)_*}(\F_3)/(\alpha_1, \alpha_2, \beta)\cdot\{v_0v_2^2\}$, consists entirely of $d_2$-cycles and receives no differentials. Thus this survives in full to the $E_3$-page.

As we have already mentioned, the Adams $E_3$-term for $\tmf$ is periodic on $v_2^3$. Combining all of these observations proves the following identification of the Adams $E_3$-term. 

\begin{prop}\label{prop: Adams E3}
	The Adams $E_3$-term is given as a module over $\Ext_{A(1)_*}(\F_3)$ as the infinite direct sum of the following types of modules,
	\begin{enumerate}
		\item[Pattern 1'] For all $j\geq 0$, we have the modules 
			\[
			\Ext_{A(1)_*}(\F_3)/(\alpha_2, c_6\beta, c_6\alpha_1)\cdot\{v_2^{3j}, v_2^{3j+1}b_4\}
			\]
		\item[Pattern 2'] For all $j\geq 0$, we have the modules 
			\[
			\Ext_{A(1)_*}(\F_3)/(\alpha_1,\alpha_2, \beta)\cdot \{v_0v_2^{3j}b_4, v_0^2v_2^{3j+1},v_0v_2^{3j+2}, v_0^2v_2^{3j+2}b_4\},
			\]
		and the ring structure is inherited from the Adams $E_2$-term. 
	\end{enumerate}
\end{prop}

We give the Adams chart for the $E_3$-term below in Figure \ref{fig: Adams E_3 term 0-80}.  We have given two copies of this chart, one inheriting the colors from previous charts, the other where we depict Pattern 1' in black and Pattern 2' in blue.  

\begin{figure}
\centering
\begin{subfigure}{0.49\textwidth}
	\includegraphics[angle=90,height=\textheight]{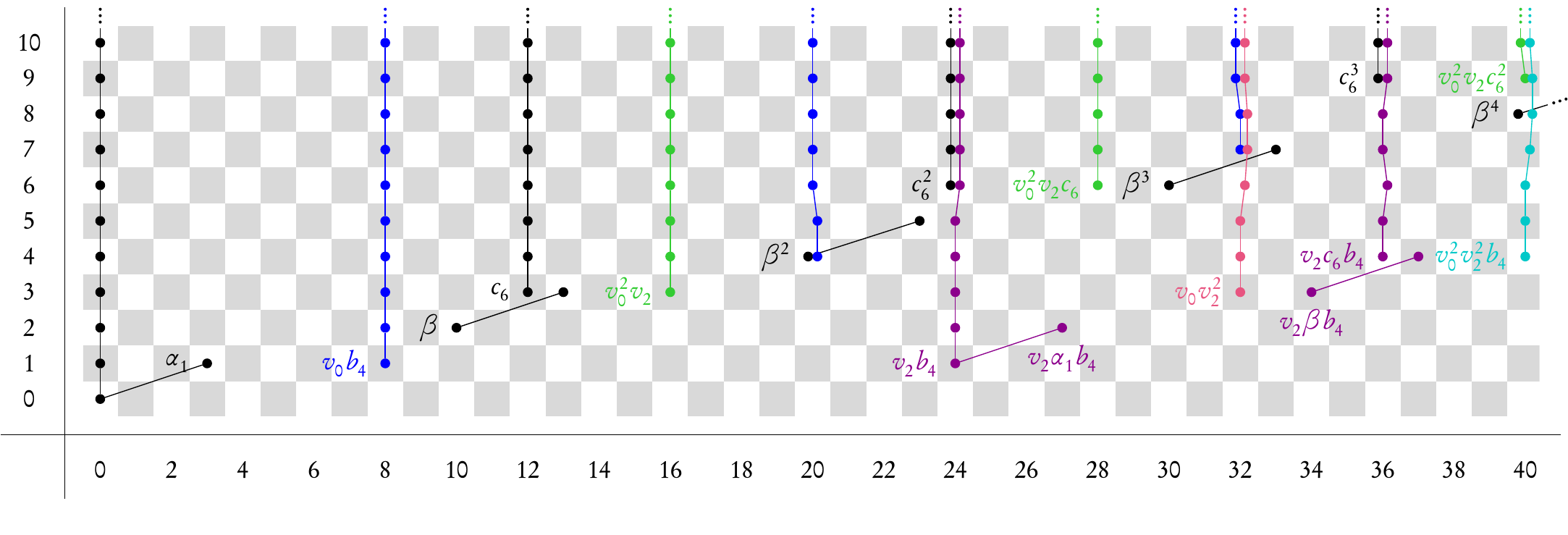}
	\caption{}
\end{subfigure}
\begin{subfigure}{0.49\textwidth}
	\includegraphics[angle=90,height=\textheight]{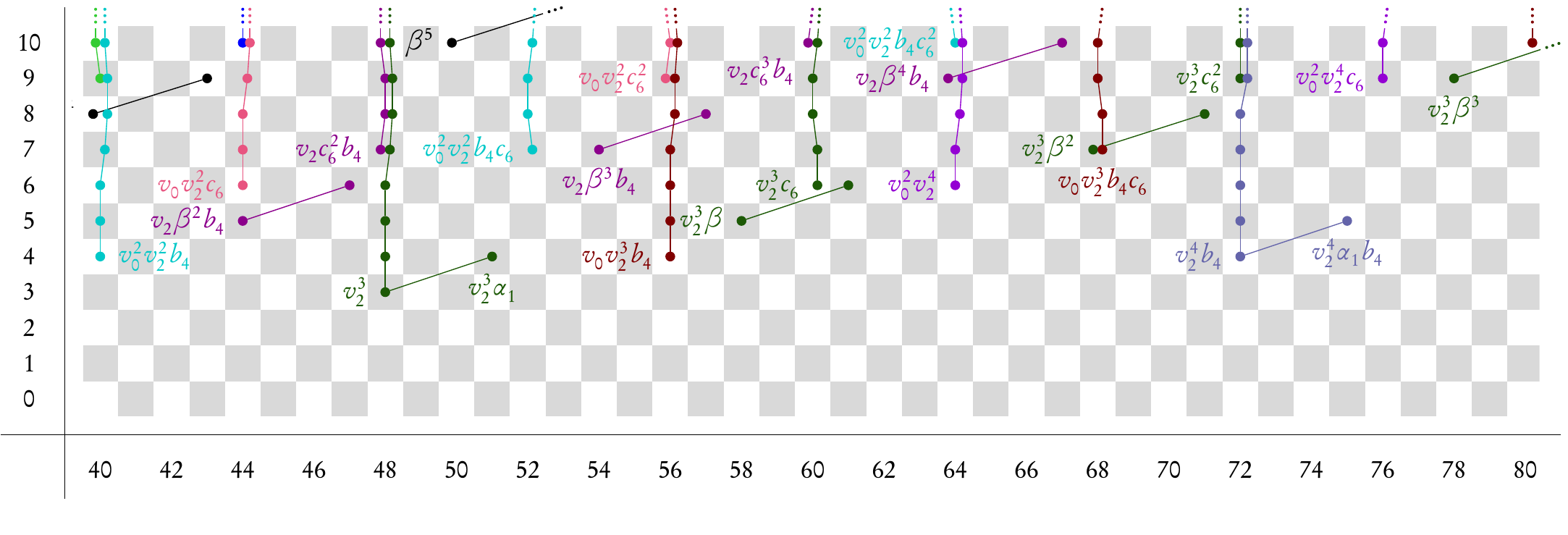}
	\caption{}
\end{subfigure}
\caption{Adams $E_3$-page in stems 0-80}\label{fig: Adams E_3 term 0-80}
\end{figure} 

\begin{figure}
\centering
\begin{subfigure}{0.49\textwidth}
	\includegraphics[angle=90,height=\textheight]{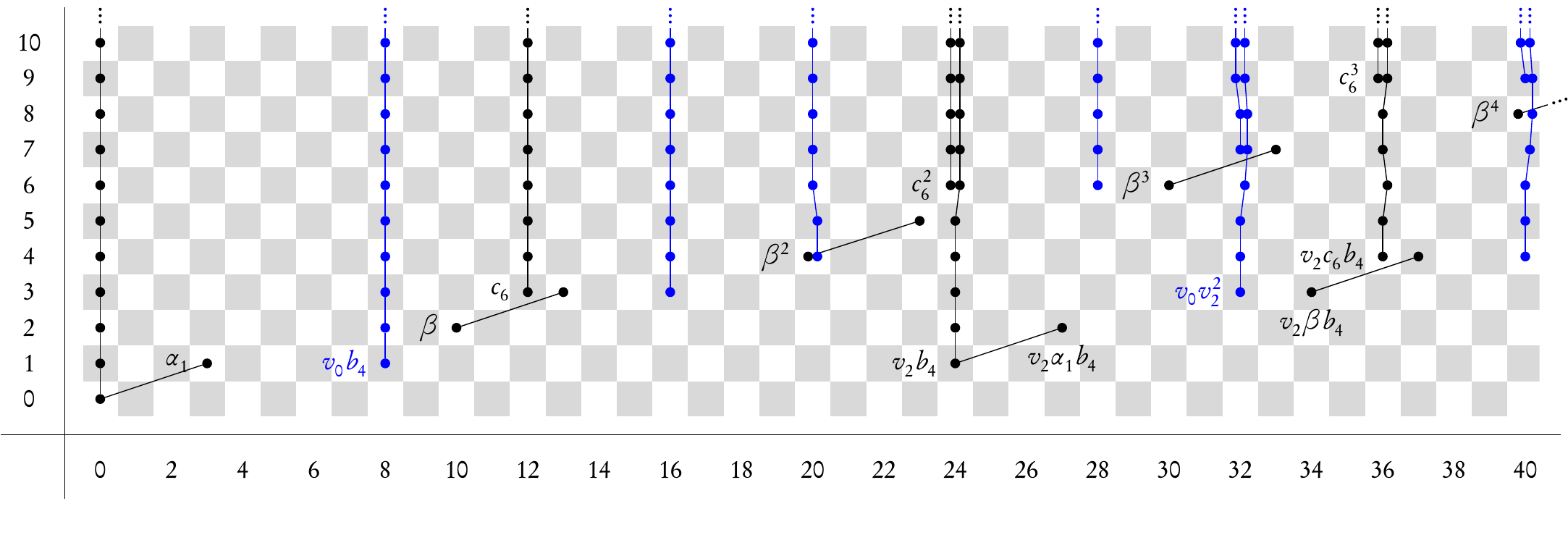}
	\caption{}
\end{subfigure}
\begin{subfigure}{0.49\textwidth}
	\includegraphics[angle=90,height=\textheight]{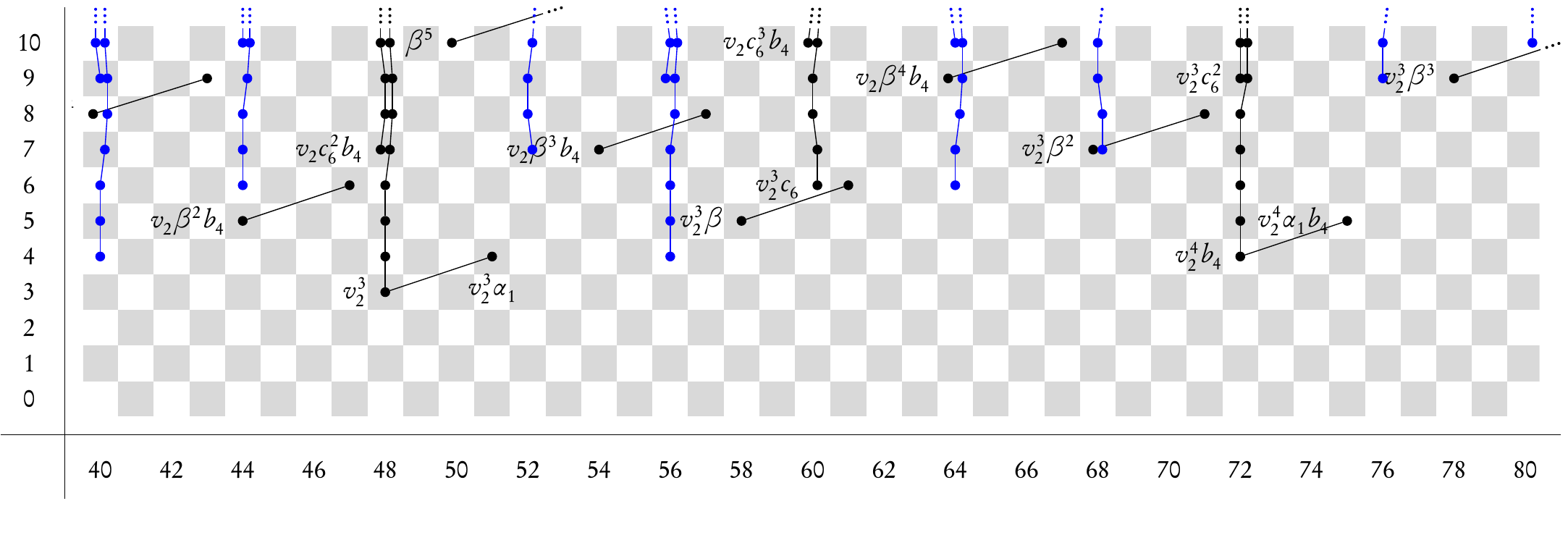}
	\caption{}
\end{subfigure}
\caption{Adams $E_3$-page in stems 0-80 with Patterns 1' and 2' highlighted}\label{fig: Adams E_3 term 0-80 modified}
\end{figure} 

As we will see later, $v_0b_4$ will detect the class $c_4$. Similarly, the class $v_0^2v_2$ will detect $c_4^2$. On the other hand, there are certain important classes in the Adams-Novikov spectral sequence which support differentials. Namely, the class $\Delta$. In the ANSS, $\Delta = v_2^{3/2}$. In the ASS, the class $v_2b_4$ corresponds to $\Delta$ while the class $v_2^3$ corresponds to $\Delta^2$. The class $v_2^4b_4$ corresponds to the class $\Delta^3$, while the class $v_2^9$ corresponds to $\Delta^6$. The reader should note that, at the $E_3$-page, we \emph{do not} have that $(b_4v_2^4)^2 = v_2^9$. In fact, $b_4v_2^4$ is not in the correct filtration for this to happen. This what makes the Adams spectral sequence more difficult than the analogous calculation in \cite{Hill_2007}. However, since $\pi_*\tmf$ is periodic on $\Delta^3$, this does suggest re-expressing the $E_3$-term in the following way. Let $M$ denote the $\Ext_{A(1)_*}(\F_3)$-module
\begin{equation}\label{eq: basic periodicity pattern}	
\begin{split}M:&= \Ext_{A(1)_*}(\F_3)/(\alpha_2, c_6\beta, c_6\alpha_1)\cdot\{v_2^{3j}, v_2^{3j+1}b_4\mid j=0,1,2\}\\ &\oplus \Ext_{A(1)_*}(\F_3)/(\alpha_1, \alpha_2, \beta)\cdot\{v_0v_2^{3j}b_4, v_0^2v_2^{3j+1}, v_0v_2^{3j+2}, v_0^2v_2^{3j+2}b_4\mid j=0,1,2\}.
\end{split}	
\end{equation}

In other words, $M$ is the collection of all the patterns from Proposition \ref{prop: Adams E3} which are generated by the listed monomials in degrees less than $|v_2^9| = 144$. The following now follows from the previous proposition.

\begin{cor}
	There is an isomorphism of $\Ext_{A(1)_*}(\F_3)$-modules 
	\begin{equation}\label{eq: Delta cubed decomposition}
	E_3(\tmf)\cong \bigoplus_{k\geq 0}M\cdot\{v_2^{9k}\}
	\end{equation}
\end{cor}

\begin{rmk}
	We will see that the classes $v_2^{9j}$ detect the elements $\Delta^{6j}$, while the classes $b_4v_2^{9j+4}$ will detect the elements $\Delta^{6j+3}$.
\end{rmk}

\begin{rmk}\label{rmk: partition of M}
	It might be easier for the reader to regard $M$ as being comprised of three pieces. Let $j\in \{0,1,2\}$ and define
	\[
	\begin{split}M_j:&= \Ext_{A(1)_*}(\F_3)/(\alpha_2, c_6\beta, c_6\alpha_1)\cdot\{v_2^{3j}, v_2^{3j+1}b_4\}\\ &\oplus \Ext_{A(1)_*}(\F_3)/(\alpha_1, \alpha_2, \beta)\cdot\{v_0v_2^{3j}b_4, v_0^2v_2^{3j+1}, v_0v_2^{3j+2}, v_0^2v_2^{3j+2}b_4\}.
	\end{split}	
	\]
	Then 
	\[
	M  = M_0\oplus M_1 \oplus M_2.
	\]
\end{rmk}

At this point the Adams $E_3$-term is ``isomorphic" to the Adams-Novikov $E_2$-term but with the elements in the ``wrong'' filtrations. All of the later differentials correspond to the usual differentials in the Adams-Novikov spectral sequence, and in fact we could deduce them from that spectral sequence. However, we try to provide arguments from first principles below.

\subsection{Higher Adams differentials}

The Adams $E_3$-term for $\tmf$ is much sparser than the $E_2$-term. This greatly reduces the possiblity of higher Adams differentials. We will now determine the $d_3$, $d_4$, and $d_6$ differentials in the Adams spectral sequence for $\tmf$. First, we make the following observation. 

\begin{prop}\label{prop: pattern 2' are pc}
	The elements in Pattern 2' of Proposition \ref{prop: Adams E3} are permanent cycles. Furthermore, this pattern receives no $d_3$-differentials. 
\end{prop} 
\begin{proof}
	Since the higher Adams differentials are linear over $\Ext_{A(1)_*}(\F_3)$, it suffices to check that the generators of Pattern 2' are permanent cycles. These are the elements $v_0v_2^{3j}b_4, v_0^2v_2^{3j+1}, v_0v_2^{3j+2}, v_0^2v_2^{3j+2}b_4$. Note that these are all in even degree, so their potential targets are in odd degree. From Proposition \ref{prop: Adams E3}, the only elements of odd degree on the $E_3$-term are of the form $v_2^{3k}\beta^\ell \alpha_1$ or $v_2^{3k}b_4\beta^\ell\alpha_1$. 
	
	Recall that $|v_2|=16$, so the degrees of the generators for Pattern 2' are congruent to $0$ or $8$ mod 16. On the other hand, the degrees of the potential targets are congruent to $13$ or $5$ mod 16. This shows the generators of Pattern 2' cannot support a differential for degree reasons. Similar considerations show that these elements cannot be the targets of any $d_3$-differentials. 
\end{proof}


Recall that the unit map 
\[
S^0\to \tmf
\]
for $\tmf$ induces a map in Ext taking the class $\beta$ to $\beta$ and $\alpha_1$ to $\alpha_1$. This is used in the calculation \cite{Bauer_2008} to derive higher Adams-Novikov differentials. We will also use it to derive higher Adams differentials. 

\begin{prop}\label{prop: d3 v2cubed = b4v2beta2alpha1}
	There is an Adams $d_3$-differential 
	\[
	d_3(v_2^3)\,\dot{=}\, b_4v_2\beta^2\alpha_1.
	\]
\end{prop}
\begin{proof}
	It is known that $\pi_{47}(\tmf)=0$. The only non-zero class in that stem on the $E_3$-term is $b_4v_2\beta^2\alpha_1$. This forces the stated differential.
\end{proof}

\begin{rmk}
	The author has made attempts to give an argument for this differential from first principles. But as of the writing of this article, he has been unable to find one. 
\end{rmk}

\begin{prop}\label{prop: d4 b4v2}
	There is an Adams $d_4$-differential 
	\[
	d_4(b_4v_2)\, \dot{=}\, \beta^2\alpha_1.
	\]
\end{prop}
\begin{proof}
	Since the classes $\alpha_1$ and $\beta$ are both in the Hurewicz image of $\tmf$, so are all the monomials $\beta^n\alpha_1^{\varepsilon_1}$. In the stable homotopy of $S^0$, the class $\beta^3\alpha_1$ is zero (see \cite[Figure 1.2.15]{greenbook}). Since the corresponding class in $E_3(\tmf)$ is not zero, it must be hit by a differential. The only possible class which could support a differential to $\beta^3\alpha_1$ is $b_4v_2\beta$. Thus we have the $d_4$-differential
	\[
	d_4(b_4v_2\beta)\, \dot{=}\, \beta^3\alpha_1.
	\]
	As the Adams differentials are linear over $\beta$, we infer
	\[
	d_4(b_4v_2)\, \dot{=}\, \beta^2\alpha_1.
	\]
\end{proof}

\begin{rmk}
	Recall that $b_4v_2$ corresponds to the class $\Delta$ in the relative Adams spectral sequence (or in the Adams-Novikov spectral sequence). In particular, this differential corresponds to the $d_2$-differential $d_2(\Delta) = \alpha\beta^2$ in \cite{Hill_2007}. In that spectral sequence, this implies the differential $d_2(\Delta^2) \, \dot{=}\, \Delta \alpha\beta^2$. However, in the ASS, the class $b_4v_2$ squares to 0, so we cannot establish such a $d_2$-differential. Rather it corresponds to the $d_3$-differential we established in Proposition \ref{prop: d3 v2cubed = b4v2beta2alpha1}. It is interesting to note these $d_2$-differentials in the relative ASS get decoupled in the ASS.
\end{rmk}

\begin{rmk}\label{rmk: no d4 differential on b4v2^4}
	One might like to think that there is the $d_4$-differential
	\[
	d_4(b_4v_2^4)\, \dot{=}\, v_2^3\beta^2\alpha_1,
	\]
	because one can multiply the differential on $b_4v_2$ to get this differential. But this is not the case since $v_2^3$ supports a shorter differential. This is an important occurrence in this spectral sequence because $b_4v_2^4$ is detecting the class $\Delta^3$ and the homotopy groups of $\tmf$ are famously periodic on $\Delta^3$.
\end{rmk}

For degree reasons, these are the only possible $d_3$ and $d_4$ differentials. We will now produce the last differential in the Adams spectral sequence. In order to do that, we will need the following observation. 

\begin{lem}\label{lem: toda bracket for b}
	The class $b:=b_4v_2\alpha_1$ is given on the $E_5$-page as the following Massey product
	\[
	b = \langle \beta^2, \alpha_1, \alpha_1\rangle.
	\]
	Thus, by Moss' Convergence Theorem, the class $b$ in $\pi_*\tmf$ is given by the corresponding Toda bracket. 
\end{lem}
\begin{proof}
	The differential $d_4(v_2b_4)=\beta^2\alpha_1$ gives a defining system for the Massey product on the $E_5$-page, and there is zero indeterminacy. Furthermore, the Toda bracket $\langle \beta^2, \alpha_1, \alpha_1\rangle$ is defined. Since there are no differentials up to the 30 stem after the $E_4$-page, there are no crossing differentials to worry about. So by Moss' convergence theorem, $b$ is given by the associated Toda bracket. 
\end{proof}

\begin{prop}[compare with \cite{Bauer_2008, Goerss_2003}]\label{prop: hidden extension b alpha1 = beta cubed}
	There is the following hidden multiplicative extension in $\pi_*\tmf$,
	\[
	b\cdot \alpha_1 \,\dot{=}\,\beta^3.
	\]
\end{prop}
\begin{proof}
	Recall that $\beta$ is given by the Toda bracket $\langle \alpha_1, \alpha_1, \alpha_1\rangle$. So, by the first juggling lemma (cf. \cite[Appendix 1]{greenbook}), it follows that 
	\[
	b\cdot \alpha_1 = \langle \beta^2, \alpha_1, \alpha_1\rangle \alpha_1 = \beta^2\langle \alpha_1, \alpha_1, \alpha_1\rangle = \beta^3.
	\]
\end{proof}

\begin{cor}\label{cor: d6 differential}
	The class $\beta^5$ is 0 in the homotopy groups of $\tmf$. Thus, there is a $d_6$-differential
	\[
	d_6(v_2^3\alpha_1)\, \dot{=}\, \beta^5.
	\]
\end{cor}
\begin{proof}
	Using the multiplicative extension of the previous proposition, we have 
	\[
	\beta^5 = \beta^2\beta^3 = \beta^2b\alpha_1.
	\]
	Since $\beta^2\alpha_1=0$, we have that $\beta^5=0$ in $\pi_*\tmf$. This forces the claimed differential.
\end{proof}

Thus far, we have only produced higher Adams differentials on generators in the submodules $M_0$ and $M_1$ of $M$ (see Remark \ref{rmk: partition of M}). There remains the summand $M_2$. Multiplication by $v_2^3$ propagates the differential from Proposition \ref{prop: d3 v2cubed = b4v2beta2alpha1} to a differential
\[
d_3(v_2^6) \,\dot{=}\, v_2^4\beta^2\alpha_1 b_4.
\]
Unfortunately, as in Remark \ref{rmk: no d4 differential on b4v2^4}, we cannot multiply by $v_2^3$ to infer differentials on $v_2^7b_4$ or $\alpha_1 v_2^7b_4$. However, \cite[Theorem 1.1]{Moss_1970} tells us that we have a Leibniz type rule for differentials on Toda brackets. We will use this to derive the desired differentials. 

\begin{lem}\label{lem: b_4v_2^4 pc}
	The class $b_4v_2^4$ is a permanent cycle.
\end{lem}
\begin{proof}
	The only possible differential that $b_4v_2^4$ could have supported was a $d_4$-differential to $v_2^3\beta\alpha_1$. But we have already shown that this class supports a $d_6$-differential in Corollary \ref{cor: d6 differential}.
\end{proof}

\begin{prop}
	There is the following Massey product on the $E_4$-page of the Adams spectral sequence for $\tmf$,
	\[
	v_2^7b_4 = \langle \beta^2\alpha_1, v_2^4b_4, v_2b_4\rangle_{E_4}
	\]
	Consequently, we have that 
	\[
	d_4(v_2^7b_4) \,\dot{=}\, \langle \beta^2\alpha_1, v_2^4b_4,\beta^2\alpha_1\rangle \,\dot{=}\, v_2^6\beta^2\alpha_1.
	\]
\end{prop}
\begin{proof}
	A defining system for the first Massey product arises from the differential $d_3(v_2^6) = \beta^2\alpha_1v_2^4b_4$. There is no indeterminacy, hence we have an equality. Since $\beta^2\alpha_1$ and $v_2^4b_4$ are permanent cycles, the differential is an immediate consequence of \cite[Theorem 1.1]{Moss_1970}. 
\end{proof}

We would like to derive a $d_6$-differential on $v_2^7b_4$ to $\beta^5v_2^4b_4$. The argument will be similar to the one found in Corollary \ref{cor: d6 differential}. 

\begin{prop}
	The class $b':= v_2^6\alpha_1$ is given on the $E_4$-page by the Massey product 
	\[
	b' = \langle v_2^4b_4\beta^2, \alpha_1, \alpha_1\rangle.
	\]
	By Moss' convergence theorem this class survives to a class $b'$ in $\pi_*\tmf$ given by the corresponding Toda bracket. Consequently, there is the hidden extension
	\[
	b'\cdot \alpha_1 \,\dot{=}\, \beta^3v_2^4b_4
	\]
	in $\pi_*\tmf$. 
\end{prop}
\begin{proof}
	The argument is completely analogous to the one in Lemma \ref{lem: toda bracket for b}. Alternatively, we can derive this from Lemma \ref{lem: b_4v_2^4 pc} below by using a juggling theorem for Toda brackets
	\[
	v_2^4b_4\cdot b = v_2^4b_4\cdot \langle \beta^2, \alpha_1, \alpha_1\rangle\subseteq \langle v_2^4b_4\beta^2, \alpha_1, \alpha_1\rangle. 
	\]
	Keep in mind these are Toda brackets in $\pi_*\tmf$. We can do this since $v_2^4b_4$ is a permanent cycle. The Toda bracket on the right hand side has indeterminacy in (see \cite[Proposition 5.7.2(b)]{Kochman:1996aa}) 
	\[
	v_2^4b_4\beta^2\cdot \pi_{7}\tmf +  \pi_{83}\tmf\cdot \alpha_1.
	\]
	It is seen on the $E_4$-page that $\pi_7\tmf = \pi_{83}\tmf =0$. Thus there is no indeterminacy. 
\end{proof}

\begin{prop}
	The class $\beta^5v_2^4b_4$ is 0 in $\pi_*\tmf$. Thus there is the $d_6$-differential
	\[
	d_6(\alpha_1v_2^7b_4) = \beta^5 v_2^4b_4
	\]
\end{prop}

Thus we have analogs of the differentials in Proposition \ref{prop: d4 b4v2} and Corollary \ref{cor: d6 differential} in the summand $M_2$ of $M$. 

%



Now we move on to showing $v_2^9$ is a permanent cycle. 

\begin{prop}\label{prop: v_2^9 is a pc}
	The class $v_2^9$ is a permanent cycle. 
\end{prop}
\begin{proof}

Note that $v_2^9$ is in stem 144. Thus, any higher Adams differential supported by $v_2^9$ must live in odd degree. The only odd degree elements in the $E_3$-term live in Pattern 1' in Proposition \ref{prop: Adams E3}. In fact, since the generators of these patterns, $v_2^{3j}$ or $v_2^{3j+1}b_4$, are in even stems, it follows that the only possible targets of a differential on $v_2^9$ are of the form $v_2^{3j}\beta^k\alpha_1$ or $v_2^{3j+1}b_4\beta^k\alpha_1$. By examining the degrees of these elements, we see that the only ones in degree 143 are $v_2^7b_4\beta^2\alpha_1$ or $\beta^{14}\alpha_1$. 

The first option could be the target of a $d_3$-differential, but it follows from Proposition \ref{prop: d3 v2cubed = b4v2beta2alpha1}, Lemma \ref{lem: b_4v_2^4 pc}, and multiplicativity that $v_2^7b_4\beta^2\alpha_1$ supports a $d_6$-differential targeting $v_2^4b_4 \beta^7$.  If, however, $v_2^4b_4\beta^7$ where hit by a $d_r$-differential with $3\leq r<6$, then we would not be able to exclude $v_2^7b_4\beta^2\alpha_1$ as a target of a $d_3$-differential. However, this cannot happen since the only elements in stem 143 on the $E_3$-page are $v_2^7b_4\beta^2\alpha_1$ and $\beta^14\alpha_1$. 

The other option, $\beta^{14}\alpha_1$ could be the target of a $d_{20}$-differential, but the target is hit by an earlier $d_3$-differential supported by $\beta^{12}v_2b_4$. Thus $v_2^9$ is a permanent cycle. 

\end{proof}

We can now derive that the Adams spectral sequence for $\tmf$ collapses at $E_7$. Towards this end, let $\overline{M}$ denote subquotient obtained from $M$ by incorporating the $d_3$ to $d_6$-differentials. Then we have the decomposition

\begin{equation}\label{eq: Delta cubed decomposition on E7}
	E_7(\tmf)\cong \bigoplus_{k\geq 0}\overline{M}\cdot\{v_2^{9k}\}
\end{equation}

Below, in Figure \ref{fig: Mbar}, is a depiction of $\overline{M}$. The reader should note the chart for $\overline{M}$ in stems $0\leq t-s< 72$ and in stems $72\leq t-s<144$ are identical up to a shift in Adams filtration. This is explained by the fact that $v_2^4b_4$ will detect the class $\Delta^3\in \pi_{72}\tmf$ (Corollary \ref{cor: detecting powers of Delta}), which is the famous periodicity generator.


\begin{figure}
	\centering
	\begin{subfigure}{0.49\textwidth}
		\includegraphics[angle=90, height=\textheight]{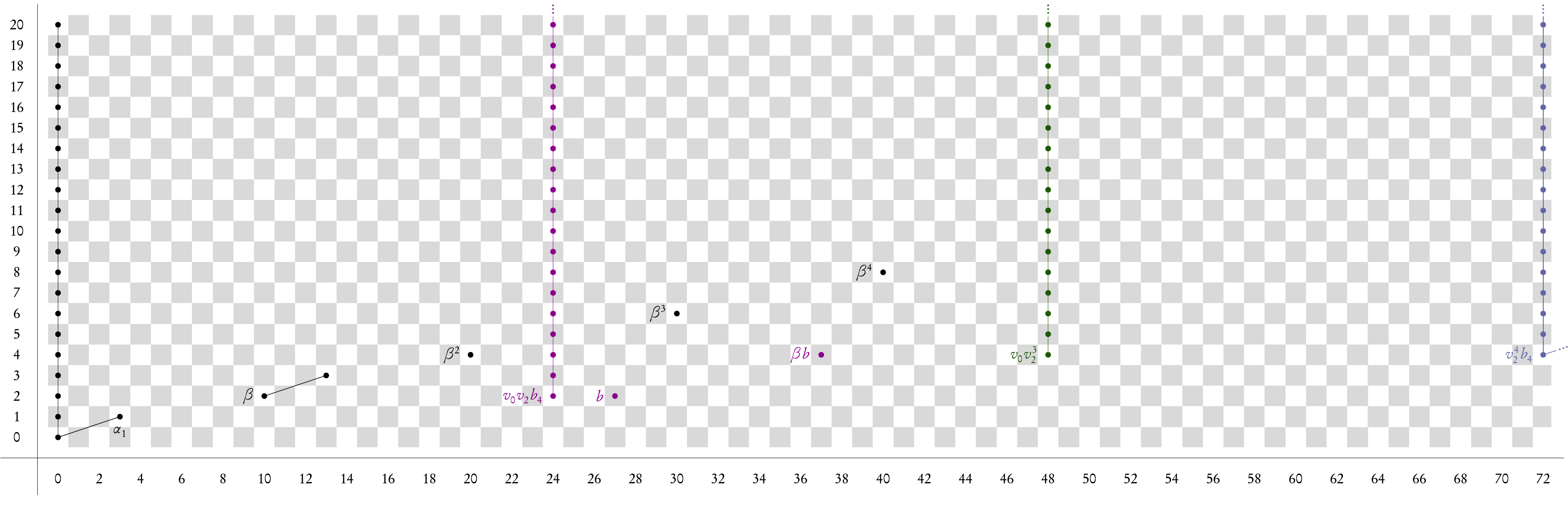}
	\end{subfigure}
	\begin{subfigure}{0.49\textwidth}
		\includegraphics[angle=90, height=\textheight]{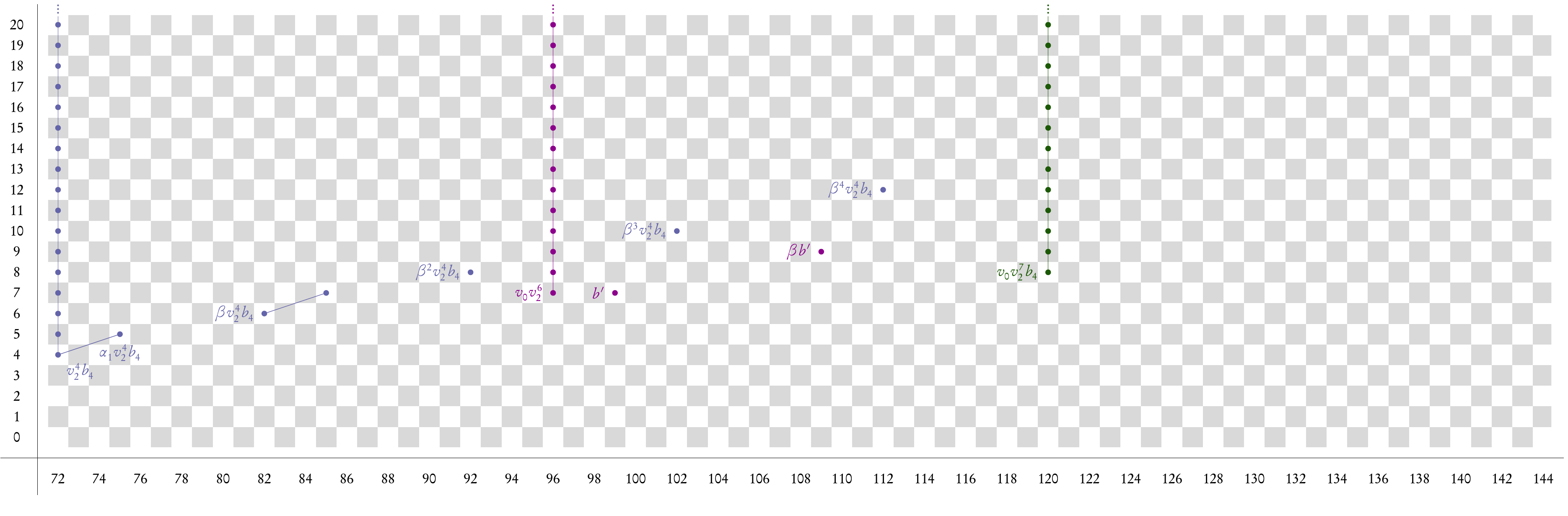}
	\end{subfigure}
	\caption{Depiction of $\overline{M}$ modulo elements of Pattern 2'}\label{fig: Mbar}
\end{figure}

\begin{prop}
	The submodule $\overline{M}\cdot\{1\}$ consists of permanent cycles. Consequently, the Adams spectral sequence for $\tmf$ collapses at $E_7$.
\end{prop}
\begin{proof}
		Since the Adams $E_7$-term is periodic on $v_2^9$, it is sufficient to check that there are no differentials supported by the indecomposable elements of $\overline{M}$. It is easily seen from Figures \ref{fig: Mbar} that the elements originating from Pattern 1' in the $E_3$-term cannot support $d_7$-differentials whose target is also an element originating from Pattern 1'. This leaves the possibility of the indecomposable classes supporting differentials into pattern 2'.
		
		Since Pattern 2' is concentrated in even degrees, this excludes the possibility of the indecomposable classes $v_0v_2b_4, v_0v_2^3, v_0v_2^6, v_0v_2^7b_4$ from supporting differentials; hence these classes are permanent cycles. The classes $b, \beta b, b'$ and $\beta b'$ may support differentials into Pattern 2'. Since the differentials are linear over $\beta$, we are reduced considering the classes $b$ and $b'$. 
		
		Observe that generators of Pattern 2' are all in stems congruent to 0 mod 8. On the other hand, $b$ and $b'$ are in stems 27 and 99 respectively. Observe that these are congruent to 3 modulo 8. Thus, the degree of a possible target of a differential supported by $b$ or $b'$ must lie in degree congruent to 2 mod 8. Since all of the elements of Pattern 2' are in stems congruent to 0 modulo 8, $b$ and $b'$ cannot support differentials into Pattern 2'. As we have already mentioned, $b$ and $b'$ also cannot support differentials into Pattern 1'. Thus $b$ and $b'$ are permanent cycles. 
		
		This shows that all of the elements of $\overline{M}$ are permanent cycles. Since $v_2^9$ is a periodicity generator for $E_7(\tmf)$, it follows from Proposition \ref{prop: v_2^9 is a pc} and the fact that the ASS for $\tmf$ is multiplicative that $E_7=E_\infty$. 
\end{proof}


We provide the Adams $E_3$-term along with all higher differentials as well as a chart for the $E_7 = E_\infty$-term below in Figures \ref{fig: Adams E_3 with differentials} and \ref{fig: Adams E_infty with differentials}

\begin{figure}
\centering
\begin{subfigure}{0.49\textwidth}
	\includegraphics[angle=90,height=\textheight]{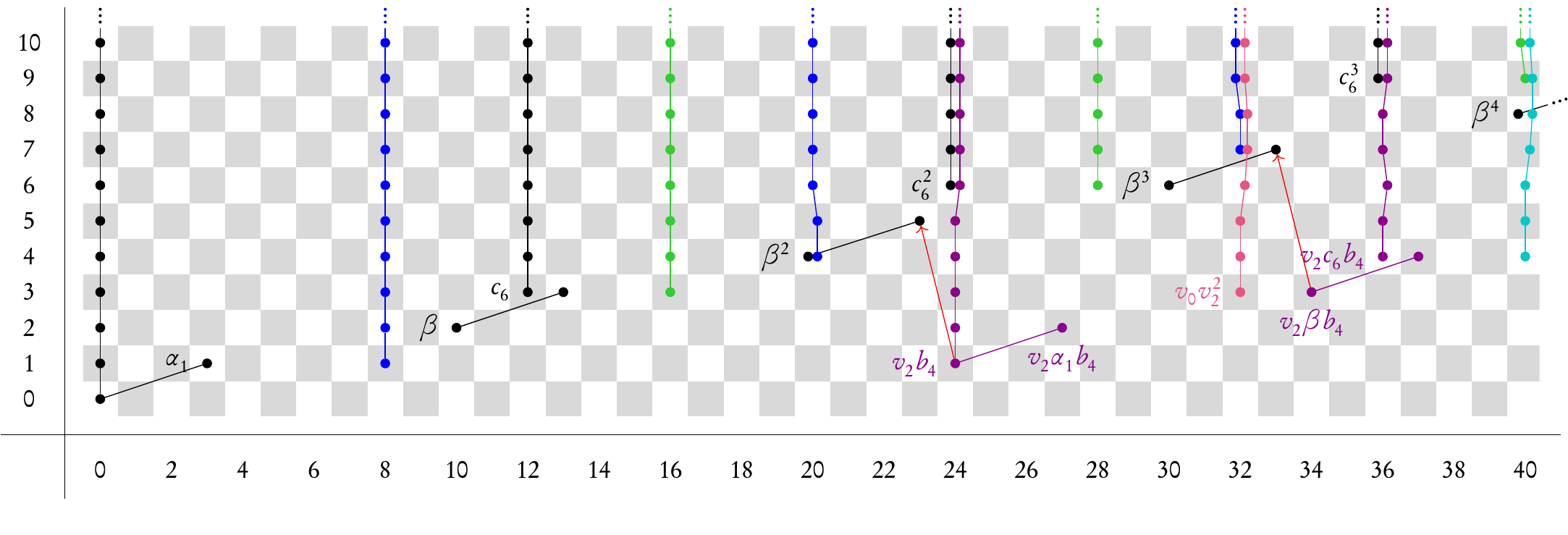}
	\caption{}
\end{subfigure}
\begin{subfigure}{0.49\textwidth}
	\includegraphics[angle=90,height=\textheight]{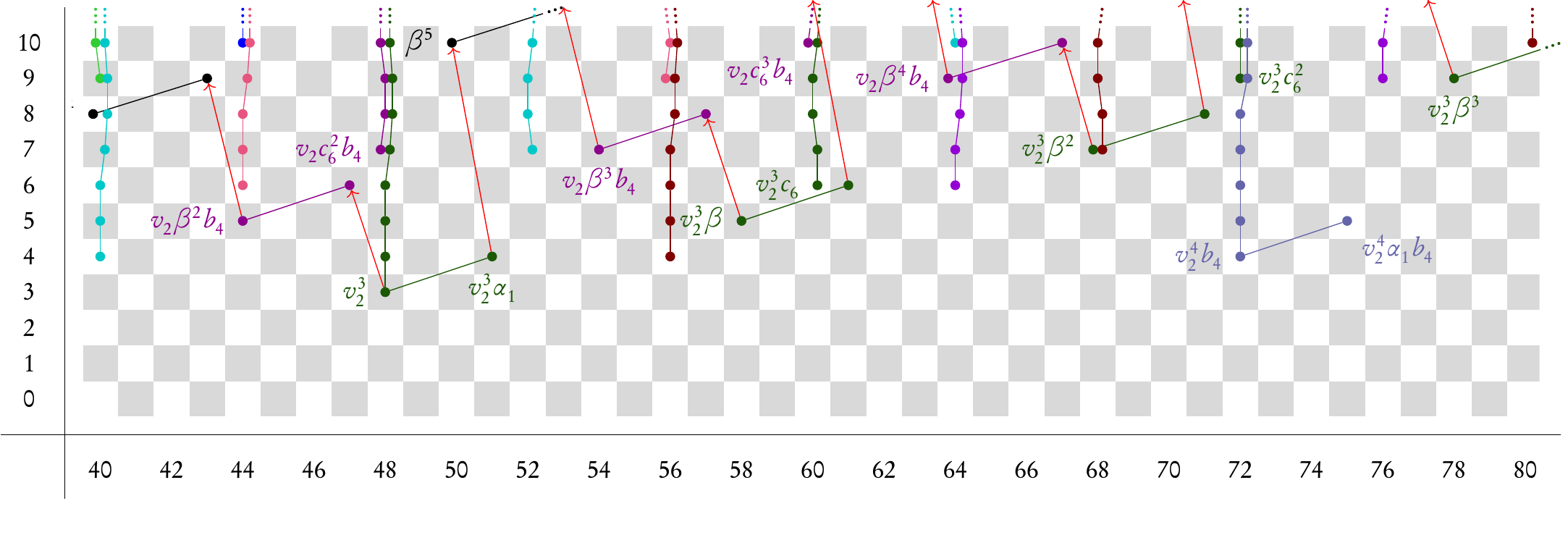}
	\caption{}
\end{subfigure}
\caption{Adams $E_3$-page in stems 0-80 with $d_3$ to $d_6$-differentials}\label{fig: Adams E_3 with differentials}
\end{figure} 

\begin{figure}
\centering
\begin{subfigure}{0.49\textwidth}
	\includegraphics[angle=90,height=\textheight]{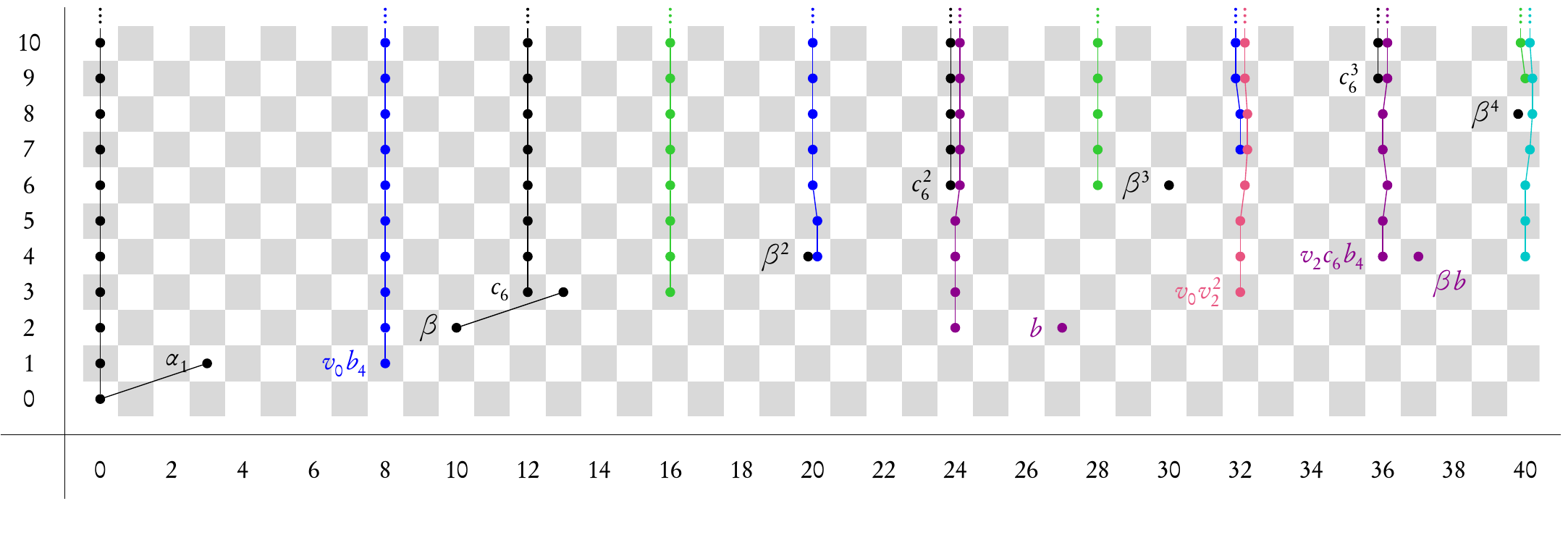}
	\caption{}
\end{subfigure}
\begin{subfigure}{0.49\textwidth}
	\includegraphics[angle=90,height=\textheight]{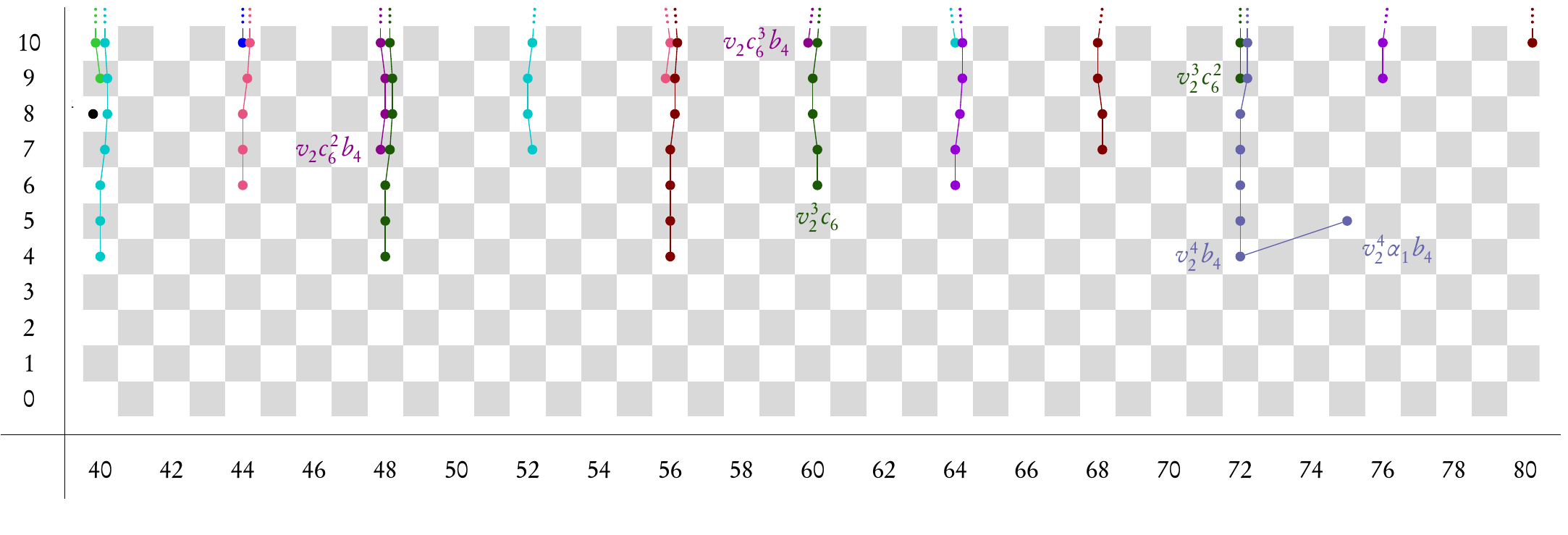}
	\caption{}
\end{subfigure}
\caption{Adams $E_\infty$-page in stems 0-80}\label{fig: Adams E_infty with differentials}
\end{figure} 

\subsection{Hidden extensions}\label{subsec: hidden extensions}

In the previous subsection we showed that the Adams spectral sequence for $\tmf$ collapses at $E_7$ and we completely computed this page via \eqref{eq: Delta cubed decomposition on E7}. We've already established one hidden extension in Proposition \ref{prop: hidden extension b alpha1 = beta cubed}, which corresponds to the single hidden extension occurring in the Adams-Novikov spectral sequence for $\tmf$. 

However, there are several relations in $\pi_*\tmf$ which are apparent on the Adams-Novikov $E_\infty$-page appearing in the 0-line, but which are hidden from the perspective of the Adams spectral sequence. We make several observations. 

\begin{prop}\label{lem: v0b4 detects c4}
	The class $v_0b_4$ in $E_\infty$ detects the class $c_4$ in $\pi_8\tmf$ up to a unit.
	\end{prop}
	\begin{proof}
		From Theorem \ref{thm: edge homomorphism}, the modular form $c_4$ is in $\pi_*\tmf$. Since $c_4$ is of degree 8 and a torsion free class, it must be detected by a class in stem 8 in the $E_\infty$-page which supports an entire $v_0$-tower. The only such class is $v_0b_4$. 
	\end{proof} 

Looking at our chart for $E_\infty$, we find that there is a single $v_0$-tower in the $16$-stem which is generated by $v_0^2v_2$. This implies the following, 

\begin{prop}
	The class $v_0^2v_2$ detects the class $c_4^2$ in $\pi_{16}\tmf$ and we have the hidden extension $(v_0b_4)\cdot (v_0b_4)\,\dot{=}\, v_0^2v_2$ in the $E_\infty$-term.
\end{prop}

\begin{rmk}
	In light of Proposition \ref{lem: v0b4 detects c4} and Proposition \ref{prop: v_1^3 detects c6}, we will abuse notation and write $v_0b_4$ as $c_4$ and $v_1^3$ as $c_6$.
\end{rmk}

We can also say which classes are detecting the various classes involving $\Delta$ in $\pi_*\tmf$. We will rename some classes in order to give more streamlined expressions. We will rename $b_4v_2$ by $v_2^{3/2}$. Thus, for example, the class $v_2^{9/2}$ refers to $b_4v_2^4$. We will use the expression $(v_2^{3/2})^k$, when $k = 2\ell$, to mean $v_2^{3\ell}$, while when $k = 2\ell+1$ this expression stands for $b_4v_2\cdot v_2^{3\ell} = b_4v_2^{3\ell+1}$. At the moment, we have introduced this notation more for convenience, it is \emph{not} reflective of a multiplicative structure on any page of this spectral sequence. Indeed, on $E_\infty$, the square of $b_4v_2$ is 0. However, this notation is motivated by a certain hidden extension which will appear shortly. 

Note that from the results of the previous section, we have

\begin{lem}
	When $j\equiv 1,2\mod 3$, the classes $(v_2^{3/2})^j$ support a differential, and in this case the classes $v_0(v_2^{3/2})^j$ are permanent cycles. 
\end{lem}

We can determine what these classes detect in $\pi_*(\tmf)$. 

\begin{cor}\label{cor: detecting powers of Delta}
	For $j\equiv 1,2\mod 3$, the classes $v_0(v_2^{3/2})^j$ detect $3\Delta^j$ up to a unit. For $j\equiv 0 \mod 3$, the class $(v_2^{3/2})^j$ detects $\Delta^j$, up to a unit. 
\end{cor}

Because these correspond to multiples of powers of $\Delta$, this implies a family of hidden extensions. 

\begin{cor}
	In $E_\infty$, we have the following hidden extensions for every $\ell\geq 0$, 
	\[
	(v_0v_2^{3/2})\cdot v_0(v_2^{3/2})^{2\ell+1} \dot{=} v_0^2v_2^{3\ell+3}
	\]
	We have the hidden extensions for odd $j$
	\[
	(v_2^{3/2})^3\cdot (v_2^{3/2})^{3j} \dot{=} (v_2^{3/2})^{3(j+1)}.
	\]
\end{cor}
This corollary justifies our choice of notation. Finally, Theorem \ref{thm: edge homomorphism} and the famous relation of modular forms
\[
c_4^3-c_6^2 = 1728 \Delta = 2^3 3^3\Delta
\]
implies a hidden extension in the $E_\infty$-term. 

\begin{prop}
	There is a hidden extension in $E_\infty(\tmf)$ given by 
	\[
	c_4\cdot (v_0^2v_2)\, \dot{=}\, v_0^3b_4v_2+ c_6^2
	\]
\end{prop}

These hidden extensions, of course, propagate themselves throughout the $E_\infty$-term. There are no hidden extensions beyond the ones mentioned above. 

\begin{rmk}
	It is rather unsatisfying that these hidden extensions were determined by using the known multiplicative structure in $\pi_*\tmf$. It would be nice to have arguments from first principles. It would seem that this would require knowing Massey product descriptions of various classes, such as $c_4, v_2^{3/2}$, and so on. But the author was unable to find such descriptions. 
\end{rmk}



%% file: ASStmf3_final.bbl
\begin{bibdiv}
\begin{biblist}

\bib{Bauer_2008}{article}{
      author={Bauer, Tilman},
       title={Computation of the homotopy of the spectrum \texttt{tmf}},
        date={2008Feb},
        ISSN={1464-8997},
     journal={Groups, homotopy and configuration spaces (Tokyo 2005)},
         url={http://dx.doi.org/10.2140/gtm.2008.13.11},
}

\bib{BrunerRognes}{unpublished}{
      author={Bruner, Robert},
      author={Rognes, John},
       title={{The Adams spectral sequence for topological modular forms}},
        note={in preparation},
}

\bib{StableStems}{article}{
      author={C., Isaksen~Daniel},
       title={Stable stems.},
        date={2019},
        ISSN={00659266},
     journal={Memoirs of the American Mathematical Society},
      volume={262},
      number={1269},
       pages={1},
  url={http://www.library.illinois.edu.proxy2.library.illinois.edu/proxy/go.php?url=http://search.ebscohost.com.proxy2.library.illinois.edu/login.aspx?direct=true&db=edb&AN=141010619&site=eds-live&scope=site},
}

\bib{CourbesElliptiqueFormulaire}{incollection}{
      author={Deligne, Pierre},
       title={{Courbes Elliptiques: Formulaire d'apres J. Tate}},
        date={1975},
   booktitle={Modular functions of one variable iv},
      editor={Birch, B.J.},
      editor={Kuyk, W.},
      series={Lecture Notes in Mathematics},
   publisher={Springer Berlin Heidelberg},
       pages={53\ndash 73},
         url={http://dx.doi.org/10.1007/BFb0097583},
}

\bib{Goerss_2003}{article}{
      author={Goerss, Paul},
      author={Henn, Hans-Werner},
      author={Mahowald, Mark},
       title={{The Homotopy of $L_2 V(1)$ for the Prime 3}},
        date={2003},
        ISSN={2296-505X},
     journal={Categorical Decomposition Techniques in Algebraic Topology},
       pages={125\ndash 151},
         url={http://dx.doi.org/10.1007/978-3-0348-7863-0_8},
}

\bib{Henriques}{incollection}{
      author={Henriques, Andr\'e},
       title={The homotopy groups of $\tmf$ and of its localizations},
        date={2014},
   booktitle={Topological modular forms},
      editor={Douglas, Christopher~L.},
      editor={Francis, John},
      editor={Henriques, Andr\'e},
      editor={Hill, Michael~A.},
      series={Mathematical Surveys and Monographs},
      volume={201},
   publisher={American Mathematical Society (AMS)},
}

\bib{Hill_2007}{article}{
      author={Hill, Michael~A.},
       title={The $3$-local $\mathit{tmf}$-homology of {$B\Sigma_3$}},
        date={2007Dec},
        ISSN={0002-9939},
     journal={Proceedings of the American Mathematical Society},
      volume={135},
      number={12},
       pages={4075\ndash 4087},
         url={http://dx.doi.org/10.1090/S0002-9939-07-08937-X},
}

\bib{Kochman:1996aa}{book}{
      author={Kochman, Stanley~O.},
       title={{Bordism, stable homotopy, and Adams spectral sequences}},
   publisher={American Mathematical Society,},
     address={Providence, R.I. :},
        date={1996},
        ISBN={0821806009},
}

\bib{MahowaldHopkins}{incollection}{
      author={Mahowald, Mark},
      author={Hopkins, Michael~J.},
       title={From elliptic curves to homotopy theory},
        date={2014},
   booktitle={Topological modular forms},
      editor={Douglas, Christopher~L.},
      editor={Francis, John},
      editor={Henriques, Andr\'e},
      editor={Hill, Michael~A.},
      series={Mathematical Surveys and Monographs},
      volume={201},
   publisher={American Mathematical Society (AMS)},
}

\bib{MR0238929}{article}{
      author={May, J.~Peter},
       title={Matric {M}assey products},
        date={1969},
        ISSN={0021-8693},
     journal={J. Algebra},
      volume={12},
       pages={533\ndash 568},
         url={https://doi.org/10.1016/0021-8693(69)90027-1},
      review={\MR{0238929}},
}

\bib{Moss_1970}{article}{
      author={Moss, R. Michael~F.},
       title={{Secondary compositions and the Adams spectral sequence}},
        date={1970},
        ISSN={1432-1823},
     journal={Mathematische Zeitschrift},
      volume={115},
      number={4},
       pages={283\ndash 310},
         url={http://dx.doi.org/10.1007/BF01129978},
}

\bib{greenbook}{book}{
      author={Ravenel, Douglas~C.},
       title={Complex cobordism and stable homotopy groups of spheres},
      series={Pure and Applied Mathematics},
   publisher={Academic Press, Inc., Orlando, FL},
        date={1986},
      volume={121},
        ISBN={0-12-583430-6; 0-12-583431-4},
      review={\MR{860042}},
}

\bib{supplementary}{unpublished}{
      author={Rezk, Charles},
       title={Supplementary notes to math 512},
        note={unpublished notes, available at
  https://faculty.math.illinois.edu/~rezk/512-spr2001-notes.pdf},
}

\end{biblist}
\end{bibdiv}
